\documentclass{amsart}
\usepackage[pdftex]{hyperref}
\usepackage{graphicx}
\usepackage{amsmath}
\usepackage{amsthm}
\usepackage{amssymb,bbm}%
\usepackage[numbers, square]{natbib}

\newcommand{\E}{\mathbb E}

\newcommand{\R}{\mathbb{R}}
\newcommand{\N}{\mathbb{N}}
\newcommand{\T}{\mathbb{T}}

\newcommand{\KKK}{\mathbb{K}}
\newcommand{\LLL}{\mathbb{L}}
\newcommand{\FFF}{\mathbb{F}}
\newcommand{\TTT}{\mathbb{T}}

\newcommand{\Var}{\mathop{\mathrm{Var}}\nolimits}
\newcommand{\Vol}{\mathop{\mathrm{Vol}}\nolimits}

\newcommand{\spec}{\mathop{\mathrm{Spec}}\nolimits}
\newcommand{\Cov}{\mathop{\mathrm{Cov}}\nolimits}

\newcommand{\sgn}{\mathop{\mathrm{sgn}}\nolimits}
\newcommand{\conv}{\mathop{\mathrm{Conv}}\nolimits}
\newcommand{\range}{\mathop{\mathrm{Range}}}
\newcommand{\wid}{\mathop{\mathrm{Width}}\nolimits}
\newcommand{\diam}{\mathop{\mathrm{diam}}\nolimits}
\newcommand{\zon}{\mathop{\mathrm{Zon}}\nolimits}

\newcommand{\Bb}{\mathcal{B}}

\newcommand{\EE}{\mathcal{E}}

\newcommand{\eqdistr}{\stackrel{d}{=}}
\newcommand{\eqfdd}{\stackrel{f.d.d.}{=}}

\newcommand{\ind}{\mathbbm{1}}

\newcommand{\dd}{{\rm d}}
\newcommand{\eee}{{\rm e}}

\theoremstyle{plain}
\newtheorem{theorem}{Theorem}[section]
\newtheorem{lemma}[theorem]{Lemma}

\newtheorem{proposition}[theorem]{Proposition}

\theoremstyle{definition}

\newtheorem{example}[theorem]{Example}

\theoremstyle{remark}
\newtheorem{remark}[theorem]{Remark}

\newenvironment{case}[1][\textsc{Case}]
{\begin{trivlist}\item[\hskip \labelsep {\textsc{#1}}]} {\end{trivlist}}

\begin{document}

\author{Zakhar Kabluchko}
\address{Zakhar Kabluchko, Institute of Stochastics,
Ulm University,
Helmholtzstr.\ 18,
89069 Ulm, Germany}
\email{zakhar.kabluchko@uni-ulm.de}

\author{Dmitry Zaporozhets}
\address{Dmitry Zaporozhets\\
St.\ Petersburg Department of
Steklov Institute of Mathematics,
Fontanka~27,
 191011 St.\ Petersburg,
Russia}
\email{zap1979@gmail.com}

\title[Intrinsic volumes of Sobolev balls]{Intrinsic volumes of Sobolev balls with applications to Brownian convex hulls}
\thanks{D.~Zaporozhets was supported by RFBR, grant 13-01-00256, and by CRC 701 ``Spectral Structures and Topological Methods in
Mathematics'', Bielefeld}
\begin{abstract}
A formula due to Sudakov relates the first intrinsic volume of a convex set in a Hilbert space to the maximum of the isonormal Gaussian process over this set. Using this formula we compute the first intrinsic volumes of infinite-dimensional convex compact sets including unit balls with respect to Sobolev-type seminorms and ellipsoids in the Hilbert space.  We relate the distribution of the random one-dimensional projections of these sets to the distributions $S_1,S_2,C_1,C_2$ studied by Biane, Pitman, Yor [\textit{Bull.\ AMS} 38 (2001)]. We show that the $k$-th intrinsic volume of the set of all functions on $[0,1]$ which have Lipschitz constant bounded by $1$ and which vanish at $0$ (respectively, which have vanishing integral) is given by
$$
V_k = \frac{\pi^{k/2}}{\Gamma\left(\frac 32 k +1 \right)},
\text{ respectively }
V_k = \frac{\pi^{(k+1)/2}}{2\Gamma\left(\frac 32 k +\frac 32\right)}.
$$
This is related to the results of Gao and Vitale [\textit{Discrete Comput.\ Geom.} 26 (2001), \textit{Elect.\ Comm.\ Probab.} 8 (2003)]  who considered a similar question for functions with a restriction on the total variation instead of the Lipschitz constant. Using the results of Gao and Vitale we give a new proof of the formula for the expected volume of the convex hull of the $d$-dimensional Brownian motion which is due to Eldan [\textit{Elect. J. Probab.}, to appear]. Additionally, we prove an analogue of Eldan's result for the Brownian bridge.  Similarly, we show that the results on the intrinsic volumes of the Lipschitz balls can be translated into formulae for the expected volumes of  zonoids (Aumann integrals) generated by the Brownian motion and the Brownian bridge. Also, these results have discrete versions for Gaussian random walks and bridges.  Our proofs exploit Sudakov's and Tsirelson's theorems which establish a connection between the intrinsic volumes and the isonormal Gaussian process.
\end{abstract}

\maketitle

\section{Introduction and statement of main results}
\subsection{Intrinsic volumes}
For a bounded convex set $T\subset\R^n$  the \textit{intrinsic volumes} $V_0(T), \ldots, V_n(T)$ are defined as the coefficients in the Steiner formula
\begin{equation}\label{eq:steiner}
\Vol_n(T+rB_n)=\sum_{k=0}^n \kappa_{n-k} V_k(T) r^{n-k}, \quad r\geq 0,
\end{equation}
where $B_n$ denotes the $n$-dimensional unit ball, $\Vol_n$ denotes the $n$-dimensional volume,  and $\kappa_k=\pi^{k/2}/\Gamma(\frac k 2 +1)$ is the volume of $B_k$.  
Denote by $W_m$ the $m$-dimensional volume of a projection of $T$ onto a uniformly chosen random $m$-dimensional linear subspace in $\R^n$. Then, Kubota's formula states that for every $1\leq m \leq n$,
\begin{equation}\label{eq:kubota}
V_m(T)= \binom nm\frac{\kappa_n}{\kappa_m \kappa_{n-m}}\E W_m.
\end{equation}
In particular,  $V_1(T)$ coincides with the so-called mean width $\E W_1$, up to a constant factor. For an extensive account on integral geometry we refer to the books~\cite{schneider_weil_book} and~\cite{klain_rota_book}

Sudakov~\cite{vS76} (who considered the case $k=1$) and Chevet~\cite{sC76} (who considered arbitrary $k\in\N$) introduced a generalization of the intrinsic volumes to \textit{infinite-dimensional} convex sets; see also~\cite[Ch.~4, \S~9.9]{burago_zalgaller_book}. Let $H$ be a separable Hilbert space. The normalization in~\eqref{eq:steiner} is chosen so that $V_k(T)$ depends only on $T$ and not on the dimension of the surrounding space, so that the definition of  $V_k(T)$ can be extended to any finite-dimensional convex subsets of $H$ (that is, convex subsets which are contained in some finite-dimensional affine subspace of $H$).
Then, for an arbitrary convex set $T\subset H$ one defines
$$
V_k(T)=\sup_{T'} V_k(T')\in [0,+\infty],
$$
where the supremum is taken over all finite-dimensional convex subsets $T'$ of $T$.

Examples of infinite-dimensional sets for which the intrinsic volumes are known explicitly are rare. The aim of this paper is to extend the list of known examples by computing the first intrinsic volume (and, whenever possible, all intrinsic volumes) of ``Sobolev balls''. These are certain infinite-dimensional convex compact subsets of the Hilbert space $L^2=L^2[0,1]$ defined in terms of Sobolev-type seminorms
$$
f \mapsto \left(\int_0^1 |f'(t)|^p\dd t\right)^{1/p}.
$$
Our proofs exploit the relation between the intrinsic volumes and the isonormal Gaussian process.



\subsection{Sobolev balls}\label{sec:sobolev_balls}
Let us define the sets we are interested in.  Denote by $AC[0,1]$  the set of absolutely continuous, real-valued functions on $[0,1]$.  Let also $\|\cdot\|_p$ be the $L^p$-norm, where $p\in [1,\infty]$.

\begin{case}{$p\neq 1$.} Let first $p\in (1,\infty]$.  Consider the set
\begin{align*}
\KKK^p =  \{f\in AC[0,1] \colon f'\in L^p, \|f'\|_p\leq 1\}.
\end{align*}
For example, it is well-known that the set $\KKK^{\infty}$ consists of all functions on $[0,1]$ with Lipschitz constant at most $1$.
The set $\KKK^p$  contains all constant functions and hence is non-compact in $L^2$. However, if we add various  boundary conditions, we can obtain compact sets.
We will consider the following sets:
\begin{align*}
\KKK^p_{BM} &= \left\{f\in \KKK^p\colon f(0)=0\right\}, \\
\KKK^p_{CBM} &=  \left\{f\in \KKK^p\colon f(0)=f(1)=0\right\}.
\end{align*}
As we will see later, these sets correspond to the Brownian Motion (BM) and the Brownian Motion centered by its integral (CBM), respectively.
Define also the following two sets corresponding to the Brownian Bridge (BB) and the Brownian Bridge centered by its integral (CBB):
\begin{align*}
\KKK^p_{BB}  &= \left\{f\in \KKK^p\colon \int_0^1 f(s)\dd s = 0\right\},\\
\KKK^p_{CBB} &= \left\{f\in \KKK^p\colon \int_0^1 f(s)\dd s = 0, f(0)=f(1)\right\}.
\end{align*}
Let also $M[0,1]$ be the set of all non-decreasing functions on $[0,1]$ and consider the sets
$$
\LLL^p =  \KKK^p \cap M[0,1],
\quad
\LLL^p_{BM} = \KKK^p_{BM} \cap M[0,1],
\quad
\LLL^p_{BB} = \KKK^p_{BB} \cap M[0,1].
$$
It makes no sense to consider the sets $\LLL^p_{CBM}$ and $\LLL^p_{CBB}$ because these sets contain only the zero function.
\end{case}

\begin{case}{$p = 1$.}
In the case $p=1$ the above definition yields sets which are not compact in $L^2$. Instead of absolutely continuous functions we have to pass to a more broad class of functions with bounded variation. For technical reasons it will be convenient to extend the functions from $[0,1]$ to $\R$.
Let $D$ be the set of all c\'adl\'ag functions $f:\R\to\R$ which are constant on the intervals $(-\infty, 0)$ and $[1,+\infty)$. The value of the function over the first interval need not coincide with the value over the other interval.  Let $TV(f)$ be the total variation of the function $f$ on $\R$. For $p=1$ we define
$$
\KKK^1 = \{f\in D \colon TV(f)\leq 1\}.
$$
We now impose various boundary conditions on the functions from $\KKK^1$. Denote by $f(t-)=\lim_{s\uparrow t} f(s)$ the left limit of $f$ at $t\in\R$.  Let $J_f(t)=f(t)-f(t-)$ be the jump of $f$ at $t\in \R$. Define
\begin{align*}
\KKK^1_{BM}  &= \left\{f\in \KKK^1   \colon f(0-) = 0, J_f(1)=0 \right\}, \\
\KKK^1_{CBM} &=  \left\{f\in \KKK^1 \colon f(0-)=f(1)=0 \right\}.
\end{align*}
Also, we consider the sets
\begin{align*}
\KKK^1_{BB}  &= \left\{f\in \KKK^1  \colon \int_0^1 f(s)\dd s = 0, J_f(0)=J_f(1)=0 \right\},\\
\KKK^1_{CBB} &= \left\{f\in \KKK^1  \colon \int_0^1 f(s)\dd s = 0, f(0)=f(1), J_f(0)=0\right\}.
\end{align*}
Denoting by $M$ the set of monotone non-decreasing functions on $\R$, we write
$$
\LLL^1 =  \KKK^1 \cap M,
\quad
\LLL^1_{BM} = \KKK^1_{BM} \cap M,
\quad
\LLL^1_{BB} = \KKK^1_{BB} \cap M.
$$
We are always interested in the values of the functions on the interval $[0,1]$, but for technical reason, we extended the functions to the whole real line. The reader may always restrict the functions under consideration to the interval $[0,1]$, but keep in mind that after such restriction the information about the value  $f(0-)$ gets lost.  Also, note that the jump at $0$ makes a contribution to the total variation $TV(f)$.
\end{case}

\vspace*{2mm}
\noindent
\textit{Notation.} Let us agree to write $\KKK_*^p$ (respectively, $\LLL^p_*$) if we mean one of the sets introduced above, where $*\in \{B
M,CBM,BB,CBB\}$ (respectively, $*\in \{BM,BB\}$).

\vspace*{2mm}
We will consider the sets $\KKK_*^p$ and $\LLL_*^p$ as subsets of $L^2=L^2[0,1]$. The next lemma shows that the embedding  of  $\KKK_*^p$ and $\LLL_*^p$ into $L^2$ is injective. It's proof will be given in Section~\ref{subsec:isonormal}.
\begin{lemma}\label{lem:ae_equal_hence_equal}
Let $p\in [1,\infty]$. If $f\in \KKK_*^p$ and $g\in \KKK_*^p$ are equal Lebesgue-a.e.\ on $[0,1]$, then they are equal everywhere on $[0,1]$ (for $p\neq 1$) or on $\R$ (for $p=1$).
\end{lemma}

The next lemma can be established by standard methods.
\begin{lemma}\label{lem:sobolev_balls_are_compact}
The sets $\KKK_*^p$ and $\LLL_*^p$ are compact, convex subsets of $L^2$ for all $p\in [1,\infty]$ and all admissible values of $*$.
\end{lemma}

\subsection{Main results on intrinsic volumes}\label{subsec:main_results}
The only known result computing explicitly intrinsic volumes of infinite-dimensional convex bodies seems to be the following theorem due to~\citet{GV01} and~\citet{fG03}.
\begin{theorem}\label{theo:gao_vitale}
For every $k\in\N$ it holds that
\begin{equation}\label{eq:gao_vitale_theo}
V_k(\LLL_{BM}^{1})
=
\frac{\kappa_k}{k!}
=
\frac{\pi^{k/2}}{\Gamma\left(\frac{k}{2}+1\right) k!},
\quad
V_k(\LLL_{BB}^{1})
=\frac{\kappa_{k+1}}{2k!}
=\frac{\pi^{(k+1)/2}}{2\Gamma\left(\frac{k}{2}+\frac 32\right) k!}.
\end{equation}
\end{theorem}
In fact, Gao and Vitale~\cite{GV01,fG03} stated their results in slightly different terms. They considered the Wiener spiral (introduced by Kolmogorov~\cite{aK40})
and the Brownian bridge spiral,
$$
\{\ind_{[0,t]}(\cdot) \colon t\in [0,1]\}\subset L^2
\;\;\;
\text{ and }
\;\;\;
\{\ind_{[0,t]}(\cdot) - t\colon t\in [0,1]\}\subset L^2,
$$
and computed the intrinsic volumes of the closed convex hulls of these sets. It is not difficult to see that these closed convex hulls are in fact isometric to $\LLL_{BM}^{1}$ and $\LLL_{BB}^{1}$.

We will complement Theorem~\ref{theo:gao_vitale} (which deals with TV-balls, $p=1$) by proving a similar result for Lipschitz balls, $p=\infty$.
\begin{theorem}\label{theo:V_k_Lipschitz_balls}
For every $k\in\N$ it holds that
\begin{equation}
V_k(\KKK_{BM}^{\infty}) = \frac{\pi^{k/2}}{\Gamma\left(\frac 32 k +1 \right)},
\quad
V_k(\KKK_{BB}^{\infty}) = \frac{\pi^{(k+1)/2}}{2\Gamma\left(\frac 32 k +\frac 32\right)}.
\end{equation}
Also, $V_k(\LLL_{BM}^{\infty}) = 2^{-k} V_k(\KKK_{BM}^{\infty})$ and $V_k(\LLL_{BB}^{\infty}) = 2^{-k} V_k(\KKK_{BB}^{\infty})$.
\end{theorem}
The proof of Theorem~\ref{theo:V_k_Lipschitz_balls} will be given in Section~\ref{subsec:Lipschitz_balls}. Theorems~\ref{theo:gao_vitale} and~\ref{theo:V_k_Lipschitz_balls} have interesting probabilistic consequences which will be discussed in Sections~\ref{subsec:gao_vitale_eldan} and~\ref{subsec:brownian_zonoids}.

For the sets $\KKK^p_*$ and $\LLL^p_*$ with general $p\in [1,\infty]$, we will compute only the first intrinsic volume. To state this result, let $\{W(t)\colon t\in [0,1]\}$ be a standard Brownian motion. Consider the following Gaussian processes on $[0,1]$ (which are the Brownian motion, the Brownian bridge, the centered Brownian motion and the centered Brownian bridge):
\begin{align}
&X_{BM}(t)  =  W(t),  \label{eq:X_*_def_1} \\
&X_{CBM}(t) =  W(t)-\int_0^1 W(s)\dd s, \label{eq:X_*_def_2} \\
&X_{BB}(t)  =  W(t)-tW(1), \label{eq:X_*_def_3}\\
&X_{CBB}(t) =  W(t)- tW(1) - \int_0^1 (W(s)-sW(1))\dd s. \label{eq:X_*_def_4}
\end{align}
These processes are special cases of the Gaussian free field on $[0,1]$ with suitable boundary conditions.
\begin{theorem}\label{theo:V_1_sobolev}
Let $p\in [1,\infty]$ and $\frac 1p + \frac 1q = 1$ Then, for every admissible $*$ it holds that
$$
V_1(\KKK_*^p) = \sqrt {2\pi}\, \E \|X_*\|_q, \quad
V_1(\LLL_*^p) = \sqrt {2\pi}\, \E \|\max(X_*,0)\|_q.
$$
\end{theorem}
We will state and prove a more general result, Theorem~\ref{prop:mean_width_sobolev}, in Section~\ref{sec:V_1_sobolev_balls}.  In the special cases $p=1$ and $p=\infty$  we can use Theorem~\ref{theo:V_1_sobolev} to obtain explicit results.
\begin{proposition}\label{prop:V1_p_1}
In the case $p=1$ we have
\begin{align}
&V_1(\KKK_{BM}^{1}) = \pi,&
&V_1(\KKK_{BB}^{1}) = \pi \log 2, \label{eq:V_1_p_1_line1}\\
&V_1(\LLL_{BM}^{1}) = 2,&
&V_1(\LLL_{BB}^{1}) = \frac {\pi}{2}. \label{eq:V_1_p_1_line2}
\end{align}
\end{proposition}
\begin{proposition}\label{prop:V1_p_infty}
In the case $p=\infty$ we have
\begin{align}
&V_1(\KKK_{BM}^{\infty})  = 2 V_1(\LLL_{BM}^{\infty}) = \frac 43,
&
&V_1(\KKK_{BB}^{\infty})  = 2 V_1(\LLL_{BB}^{\infty}) = \frac{\pi}{4},\label{eq:V_1_p_infty_line1}\\
&V_1(\KKK_{CBM}^{\infty}) = \frac 1{6} (2\sqrt 3 + \log(2+ \sqrt 3)),
&
&V_1(\KKK_{CBB}^{\infty}) = \frac {2}{\sqrt 3}. \label{eq:V_1_p_infty_line2}
\end{align}
\end{proposition}
The formula for $V_1(\KKK_{BB}^{\infty})$ was mentioned in~\cite[Example~1]{tsirelson_1}. Note  that~\eqref{eq:V_1_p_1_line2} is a special case of Theorem~\ref{theo:gao_vitale}, whereas~\eqref{eq:V_1_p_infty_line1} is a special case of Theorem~\ref{theo:V_k_Lipschitz_balls}.

\vspace*{2mm}
For $p=2$, the Sobolev balls reduce to ellipsoids in the Hilbert space with half-axes equal to either $1,\frac 12, \frac 13,\ldots$ or $1, \frac 13, \frac 15, \ldots$. These ellipsoids will be studied in Section~\ref{sec:V_1_hilbert_ellipsoids}.


\subsection{Sudakov's and Tsirelson's theorems}\label{subsec:sudakov_thm}
The main tool in the proof of Theorem~\ref{theo:V_1_sobolev} is a formula due to Sudakov~\cite{vS76}.  It establishes a link between the first intrinsic volume and the supremum of the isonormal process. Recall that the \textit{isonormal process} over a separable Hilbert space $H$ is a mean zero Gaussian process $\{\xi(h)\colon h\in H\}$ having the covariance function
$$
\Cov (\xi(h), \xi(g)) = \langle h, g\rangle.
$$
\begin{theorem}[Sudakov]\label{theo:sudakov}
For every convex set $T\subset H$ it holds that
\begin{equation}\label{2041}
V_1(T)= \sqrt{2\pi}\,\E\,\sup_{h\in T} \xi(h).
\end{equation}
\end{theorem}

Tsirelson~\cite{tsirelson2} generalized the previous theorem to all intrinsic volumes as follows. Consider $k$ independent copies  $\{\xi_i(h)\colon h\in H\}$, $1\leq i \leq k$, of the isonormal process. The $k$-dimensional spectrum of a compact convex set $T\subset H$ is a random set
$$
\spec_k T = \{(\xi_1(h), \ldots, \xi_k(h))\colon h\in T\}\subset \R^k.
$$
Recall that the set $T$ is called a GB-set if there is a version of the isonormal process over $T$ which has bounded sample paths. It is known that the GB-property is equivalent to $V_1(T)<\infty$ in which case we also have $V_k(T)<\infty$ for all $k\in\N$; see~\cite{sC76}.
\begin{theorem}[Tsirelson]\label{theo:tsirelson_spectrum}
For every $k\in\N$ and every compact convex $GB$-set $T\subset H$ it holds that
\begin{equation}\label{2042}
V_k(T)= \frac{(2\pi)^{k/2}}{k!\kappa_k} \E\,\Vol_k(\spec_k T).
\end{equation}
\end{theorem}


\begin{remark} To see that Theorem~\ref{theo:tsirelson_spectrum} generalizes Theorem~\ref{theo:sudakov} note that in the case when $k=1$ the spectrum $\spec_1 T$ is just the range
of the process $\{\xi(h)\colon h\in T\}$ and Theorem~\ref{theo:tsirelson_spectrum} states that
$$
V_1(T)=\sqrt{\frac\pi 2}\, \E \left(\sup_{h\in T}\xi(h)-\inf_{h\in T}\xi(h)\right).
$$
However, since the processes  $\xi$ and $-\xi$ have the same distribution, it holds that
$
\E\inf_{h\in T}\xi(h) = - \E\sup_{h\in T}\xi(h)
$
and we recover~\eqref{2041}.
\end{remark}
\begin{remark}[On separability]
In order to have a well-defined supremum in~\eqref{2041}, it is tacitly assumed in Theorem~\ref{theo:sudakov} that we are dealing with the separable modification of $\xi$; see Proposition~2.6.1 in~\cite{bogachev_book} for the proof of its existence. In Theorem~\ref{theo:tsirelson_spectrum} a separable modification is not sufficient and one tacitly assumes that one is dealing with the so-called natural modifications of $\xi_1,\ldots,\xi_k$;  see~\cite[Proposition~2.6.4]{bogachev_book} for the proof of their existence in the case of a GB-set $T$. We will have no problems with separability since the sets $\KKK_*^p$ and $\LLL_*^p$ have the GC-property (meaning that the isonormal process has a version with continuous sample paths over these sets); see Lemma~\ref{lem:continuity}.
\end{remark}




\subsection{Applications to Brownian convex hulls}\label{subsec:gao_vitale_eldan}
Combining Tsirelson's Theorem~\ref{theo:tsirelson_spectrum} with the results of Section~\ref{subsec:main_results} it is possible to obtain interesting probabilistic consequences. The main idea here is that the convex hull of a Brownian motion in $\R^k$ can be viewed as a projection of the convex hull of the Wiener spiral onto a ``uniformly chosen'' random $k$-dimensional linear subspace of $L^2$. A precise formulation of this statement is given by Tsirelson's Theorem~\ref{theo:tsirelson_spectrum} (which can be seen as an infinite-dimensional analogue of the Kubota's formula~\eqref{eq:kubota}). This allows to establish a connection between the $k$-th intrinsic volume of the convex hull of the Wiener spiral and the expected volume of the convex hull of a Brownian motion in $\R^k$.

Let $\{W(t)\colon t\geq 0\}$ be a standard Brownian motion.  The isonormal process $\{\xi(f)\colon f\in L^2\}$ is given by
$$
\left\{\xi(f)\colon f\in L^2 \right\} \eqfdd \left\{\int_0^1 f(t)\dd W(t)\colon f\in L^2\right\},
$$
where the stochastic integral is in the usual It\^{o} sense.
Let
$$
\{X_{BM}^{(k)}(t)=(W_1(t),\ldots,W_k(t))\colon t\geq 0\}
$$
be the standard $\R^k$-valued Brownian motion whose components $W_1(t),\ldots,W_k(t)$ are independent copies of $W(t)$.  Using the isometry between $\LLL_{BM}^{1}$ and the closed convex hull of the Wiener spiral,  it is easy to see that the spectrum $\spec_k(\LLL_{BM}^{1})$ has the same distribution as the closed convex hull of the $k$-dimensional Brownian path $\{X_{BM}^{(k)}\colon t\in [0,1]\}$. Combining Theorem~\ref{theo:gao_vitale} with Theorem~\ref{theo:tsirelson_spectrum} we obtain that
\begin{equation}\label{eq:E_Vol_conv_hull_BM}
\E \Vol_k(\conv\{X_{BM}^{(k)}(t)\colon 0\leq t\leq 1\})
= \frac{\kappa_k^2}{(2\pi)^{k/2}}.
\end{equation}
Here, $\conv A$ denotes the convex hull of a set $A$.  Using Kubota's formula~\eqref{eq:kubota} and the fact that the $m$-dimensional projection of a $k$-dimensional Brownian motion is an $m$-dimensional Brownian motion, we obtain a generalization of~\eqref{eq:E_Vol_conv_hull_BM} to arbitrary intrinsic volumes. Namely, for all $0\leq m\leq k$,
\begin{equation}\label{eq:E_V_m_conv_hull_BM}
\E V_m(\conv\{X_{BM}^{(k)}(t)\colon 0\leq t\leq 1\})
= \frac 1 {(2\pi)^{m/2}} \binom km \frac{\kappa_k\kappa_m}{\kappa_{k-m}}.
\end{equation}
\citet{rE12} obtained~\eqref{eq:E_Vol_conv_hull_BM} and~\eqref{eq:E_V_m_conv_hull_BM} independently, but it seems that the equivalence of his result to Theorem~\ref{theo:gao_vitale} remained unnoticed. For $m=1,2$, the result~\eqref{eq:E_V_m_conv_hull_BM} is contained in~\cite[Cor.~1.4, Prop.~1.6]{kampf_etal}, see also~\cite{biane_letac} and~\cite[Ch.~4.3, 4.4]{kampf_phd}.

Similarly, it is easy to see that the spectrum $\spec_k(\LLL_{BB}^{1})$ has the same distribution as the closed convex hull of a standard $k$-dimensional Brownian bridge $\{X_{BB}^{(k)}(t) \colon t\in [0,1]\}$. Combining Theorem~\ref{theo:V_k_Lipschitz_balls} with Theorem~\ref{theo:tsirelson_spectrum} we obtain that
\begin{equation}\label{eq:E_Vol_conv_hull_BB}
\E \Vol_k(\conv\{X_{BB}^{(k)}(t)\colon 0\leq t\leq 1\})
=\frac{\kappa_k\kappa_{k+1}}{2(2\pi)^{k/2}}.
\end{equation}
Using Kubota's formula~\eqref{eq:kubota} and the fact that an orthogonal projection of $X_{BB}$ is again a Brownian bridge, we obtain that for all $0\leq m\leq k$,
\begin{equation}\label{eq:E_V_m_conv_hull_BB}
\E V_m(\conv\{X_{BB}^{(k)}(t)\colon 0\leq t\leq 1\})
= \frac 1 {2(2\pi)^{m/2}} \binom km \frac{\kappa_k\kappa_{m+1}}{\kappa_{k-m}}.
\end{equation}
\citet{RMC09} and~\citet{RMC10} obtained~\eqref{eq:E_Vol_conv_hull_BB} and~\eqref{eq:E_V_m_conv_hull_BB} for $k=2$.

Let us also mention  discrete versions of the above results. Consider the following points in $\R^n$:
$$
P_i= (\underbrace{0,\ldots,0}_{n-i},\underbrace{1,\ldots,1}_{i}),
\;\;\;
P_i^* =  P_i - \left(\frac in, \ldots, \frac in\right),\;\;\;
0\leq i\leq n.
$$
Denote by $\TTT_{BM,n}$ the convex hull of $P_0,\ldots,P_n$ and by $\TTT_{BB,n}$ the convex hull of $P_0^*,\ldots,P_n^*$. The sets $\TTT_{BM,n}$ and $\TTT_{BB,n}$ are simplices and can be seen as the discrete analogues of $\LLL^{1}_{BM}$ and $\LLL^1_{BB}$. The next theorem is due to~\citet{GV01} and~\citet{fG03}.
\begin{theorem}\label{theo:gao_vitale_discrete}
For every $k=1,\ldots,n$ it holds that
\begin{align}
V_k(\TTT_{BM,n}) &= \frac{1}{k!} \sum_{A_{n,k}} \frac{1}{\sqrt{d_1\ldots d_k}}, \\
V_k(\TTT_{BB,n}) &= \frac{1}{k!} \sum_{A_{n,k}} \sqrt\frac{n-(d_1+\ldots+d_k)}{n d_1\ldots d_k},
\end{align}
where $A_{n,k}$ is the set of all $(d_1,\ldots,d_k)\in \N^k$ with $d_1+\ldots+d_k\leq n$.
\end{theorem}
The spectrum $\spec_k(\TTT_{BM,n})$ is the convex hull of an $n$-step Gaussian random walk in $\R^k$. Similarly, $\spec_k(\TTT_{BB,n})$ is the convex hull of an $n$-step Gaussian random walk in $\R^k$ conditioned to return to $0$. Namely,
\begin{align*}
\spec_k(\TTT_{BM,n}) & \eqdistr \conv\{0, X_{BM}^{(k)}(1), \ldots, X_{BM}^{(k)}(n)\},\\
\spec_k(\TTT_{BB,n}) & \eqdistr \conv\{0, X_{BM}^{(k)}(1), \ldots, X_{BM}^{(k)}(n)\} | \{X_{BM}^{(k)}(n)=0\}.
\end{align*}
Tsirelson's Theorem~\ref{theo:tsirelson_spectrum} combined with Theorem~\ref{theo:gao_vitale_discrete} yields the formulae for the expected volumes of these convex hulls:
\begin{align}
\E \Vol_k (\spec_k(\TTT_{BM,n})) &= \frac{\kappa_k}{(2\pi)^{k/2}} \sum_{A_{n,k}} \frac{1}{\sqrt{d_1\ldots d_k}},\\
\E \Vol_k (\spec_k(\TTT_{BB,n})) &= \frac{\kappa_k}{(2\pi)^{k/2}} \sum_{A_{n,k}} \sqrt \frac{n-(d_1+\ldots+d_k)}{n d_1\ldots d_k}.
\end{align}
Kubota's formula~\eqref{eq:kubota} allows to generalize these formulae to arbitrary intrinsic volumes. We obtain that for every $1\leq m\leq k$,
\begin{align}
\E V_m (\spec_k(\TTT_{BM,n})) &= \frac 1 {(2\pi)^{m/2}} \binom{k}{m} \frac{\kappa_k}{\kappa_{k-m}} \sum_{A_{n,m}} \frac{1}{\sqrt{d_1\ldots d_m}},\label{eq:conv_hull_rand_walk_kub_BM}\\
\E V_m (\spec_k(\TTT_{BB,n})) &= \frac 1 {(2\pi)^{m/2}} \binom{k}{m} \frac{\kappa_k}{\kappa_{k-m}} \sum_{A_{n,m}} \sqrt \frac{n-(d_1+\ldots+d_m)}{n d_1\ldots d_m}. \label{eq:conv_hull_rand_walk_kub_BB}
\end{align}
For the semiperimeter $V_1$ of the convex hull of a general (not necessarily Gaussian) two-dimensional random walk $S_1,S_2,\ldots$, \citet{spitzer_widom} and~\citet{baxter} obtained the formula
$$
\E V_1(\conv\{0,S_1,\ldots,S_n\}) = \sum_{j=1}^n \frac 1j \E \|S_j\|_2.
$$
In the Gaussian case, the right-hand side of this formula is $\sqrt{\frac{\pi}2}\sum_{j=1}^n \frac{1}{\sqrt j}$. This is equivalent to~\eqref{eq:conv_hull_rand_walk_kub_BM} with $k=2$, $m=1$.

\subsection{Applications to Brownian zonoids}\label{subsec:brownian_zonoids}
The spectrum of $\LLL_{BM}^{\infty}$ is given by the zonoid spanned by the $k$-dimensional Brownian motion $\{X_{BM}^{(k)}\colon t\in [0,1]\}$:
$$
\spec_k (\LLL_{BM}^{\infty})
\eqdistr
\left\{\int_{0}^1 X_{BM}^{(k)}(t) g(t) \dd t\colon g\in L^{\infty}[0,1], 0\leq g\leq 1\right\}.
$$
This follows from  a stochastic integral representation of the isonormal process, see Lemma~\ref{lem:isonormal_alternative_rep} below, by noting that any $f\in \LLL_{BM}^{\infty}$ can be represented as $f(t)=\int_0^t g(s) \dd s$ for some measurable function $0\leq g\leq 1$.  This random set can be interpreted as the Aumann integral of the (random) set-valued function mapping $t\in [0,1]$ to the segment $[0, X_{BM}^{(k)}(t)]$. Using Theorem~\ref{theo:tsirelson_spectrum} and Theorem~\ref{theo:V_k_Lipschitz_balls} we obtain that
$$
\E \Vol_k (\spec_k (\LLL_{BM}^{\infty})) = \frac{1}{(2\sqrt{2\pi})^k} \binom{\frac 32 k}{k}^{-1}.
$$
Using Kubota's formula~\eqref{eq:kubota} and the invariance of the Brownian motion under projections, we obtain that
$$
\E V_m (\spec_k (\LLL_{BM}^{\infty})) = \binom k m \frac{\kappa_k}{\kappa_m \kappa_{k-m}}\cdot \frac{1}{(2\sqrt{2\pi})^m} \binom{\frac 32 m}{m}^{-1}.
$$
Similarly, one shows that $\spec_k (\LLL_{BB}^{\infty})$ is the zonoid spanned by the $k$-dimensional Brownian bridge $\{X_{BB}^{(k)}\colon t\in [0,1]\}$. Using Theorem~\ref{theo:tsirelson_spectrum} and Theorem~\ref{theo:V_k_Lipschitz_balls} we obtain that
$$
\E \Vol_k (\spec_k (\LLL_{BB}^{\infty})) = \frac{\sqrt \pi}{2 (2\sqrt{2\pi})^{k}} \binom{\frac 32 k}{k}^{-1}.
$$
Using Kubota's formula~\eqref{eq:kubota} and the invariance of the Brownian bridge under projections we obtain that
$$
\E V_m (\spec_k (\LLL_{BB}^{\infty})) = \binom k m \frac{\kappa_k}{\kappa_m \kappa_{k-m}} \cdot  \frac{\sqrt \pi}{2 (2\sqrt{2\pi})^{m}} \binom{\frac 32 m}{m}^{-1}.
$$

Finally, one can also obtain  discrete versions of the above results. Denote by $\zon (v_1,\ldots,v_n)$ the zonotope spanned by a collection of vectors $v_1,\ldots,v_n$:
$$
\zon (v_1,\ldots,v_n) = \{\alpha_1 v_1+\ldots+\alpha_n v_n\colon \alpha_1,\ldots, \alpha_n \in [0,1]\}.
$$
Consider the sets
\begin{equation}\label{eq:F_BM_F_BB_def}
\FFF_{BM,n} = \zon\{P_1,\ldots,P_n\},\;\;\;
\FFF_{BB,n} = \zon\{P_1^*,\ldots,P_n^*\},
\end{equation}
The sets $\FFF_{BM,n}$ and $\FFF_{BB,n}$ are parallelotopes and can be seen as the finite-dimensional analogues of the sets $\LLL_{BM}^{\infty}$ and $\LLL_{BB}^{\infty}$; see Section~\ref{subsec:fi_di_lip_zonotopes}.
We next result complements Theorem~\ref{theo:gao_vitale_discrete}.
\begin{theorem}\label{theo:V_k_Lipschitz_balls_discrete}
For every $k=1,\ldots,n$ it holds that
\begin{align}
V_k(\FFF_{BM,n})
&=
\sum_{A_{n,k}} \sqrt{d_1 \ldots d_k},
\\
V_k(\FFF_{BB,n})
&=
\sum_{B_{n,k}} \sqrt{d_1 \ldots d_{k+1}},
\end{align}
where $A_{n,k}$ is the set of all $(d_1,\ldots,d_k)\in \N^k$ with $d_1+\ldots+d_k\leq n$ and $B_{n,k}$ is the set of all $(d_1,\ldots,d_{k+1})\in \N^{k+1}$ with $d_1+\ldots+d_{k+1} = n$.
\end{theorem}
Clearly, $\spec_k(\FFF_{BM,n})$ is the zonotope spanned by the $n$-step Gaussian random walk in $\R^k$. Similarly, $\spec_k(\FFF_{BB,n})$ is the zonotope spanned by  the $n$-step Gaussian random walk in $\R^k$ conditioned to return to the origin at time $n$. Namely,
\begin{align*}
\spec_k(\FFF_{BM,n}) & \eqdistr \zon\{X_{BM}^{(k)}(1), \ldots, X_{BM}^{(k)}(n)\},\\
\spec_k(\FFF_{BB,n}) & \eqdistr \zon\{X_{BM}^{(k)}(1), \ldots, X_{BM}^{(k)}(n)\} | \{X_{BM}^{(k)}(n)=0\}.
\end{align*}
Tsirelson's Theorem~\ref{theo:tsirelson_spectrum}, together with Theorem~\ref{theo:V_k_Lipschitz_balls_discrete}, yields the formulae for the expected volumes of these zonotopes:
\begin{align}
\E \Vol_k (\spec_k(\FFF_{BM,n})) &= \frac{k! \kappa_k}{(2\pi)^{k/2}} \sum_{A_{n,k}} \sqrt{d_1\ldots d_k},\\
\E \Vol_k (\spec_k(\FFF_{BB,n})) &= \frac{k! \kappa_k}{(2\pi)^{k/2}} \sum_{B_{n,k}} \sqrt{d_1\ldots d_{k+1}}.
\end{align}
Using Kubota's formula~\eqref{eq:kubota} one obtains a  generalization of these formulae to arbitrary intrinsic volumes. For every $1\leq m\leq k$,
\begin{align}
\E V_m (\spec_k(\FFF_{BM,n})) &= \frac {m!} {(2\pi)^{m/2}} \binom{k}{m} \frac{\kappa_k}{\kappa_{k-m}} \sum_{A_{n,m}} \sqrt{d_1\ldots d_m},\\
\E V_m (\spec_k(\FFF_{BB,n})) &= \frac {m!} {(2\pi)^{m/2}} \binom{k}{m} \frac{\kappa_k}{\kappa_{k-m}} \sum_{B_{n,m}} \sqrt{d_1\ldots d_{m+1}}.
\end{align}

\section{Intrinsic volumes of Lipschitz balls: Proof of Theorems~\ref{theo:V_k_Lipschitz_balls_discrete} and~\ref{theo:V_k_Lipschitz_balls}}\label{subsec:Lipschitz_balls}

In this section we compute the intrinsic volumes of the Lipschitz balls $\KKK_{BM}^{\infty}$, $\KKK_{BB}^{\infty}$, $\LLL_{BM}^{\infty}$, $\LLL_{BB}^{\infty}$ and their finite-dimensional analogues $\FFF_{BM,n}$ and $\FFF_{BB,n}$.

\subsection{Finite-dimensional Lipschitz zonotopes}\label{subsec:fi_di_lip_zonotopes}
Recall that $\LLL^{\infty}$ is the set of non-decreasing functions on $[0,1]$ with Lipschitz constant at most $1$; see Section~\ref{sec:sobolev_balls}.
Consider the finite-dimensional analogues of the sets $\LLL_{BM}^{\infty}$ and $\LLL_{BB}^{\infty}$:
\begin{align*}
\FFF_{BM,n} &= \{(x_1,\ldots,x_n)\in\R^{n}: x_{1},x_2 - x_{1}, \ldots, x_{n}-x_{n-1}\in [0,1]\},\\
\FFF_{BB,n} &= \{(x_1,\ldots,x_n)\in\R^{n}: x_2 - x_{1}, \ldots, x_{n}-x_{n-1}\in [0,1], x_1+\ldots+x_n=0\}.
\end{align*}
It is easy to see that these definitions are equivalent to the previous ones; see~\eqref{eq:F_BM_F_BB_def}. Our aim is to prove Theorem~\ref{theo:V_k_Lipschitz_balls_discrete} which can be restated as follows:
\begin{align}
V_k(\FFF_{BM,n})
&=
\sum_{1\leq l_1<\ldots<l_k\leq n} \sqrt{l_1(l_2-l_1)\ldots (l_k-l_{k-1})}, \label{eq:V_k_F_BM_n}\\
V_k(\FFF_{BB,n})
&=
\sum_{1\leq l_1<\ldots<l_k\leq n} \sqrt{l_1(l_2-l_1)\ldots (l_k-l_{k-1})(n-l_k)}. \label{eq:V_k_F_BB_n}
\end{align}
\begin{proof}[Proof of Theorem~\ref{theo:V_k_Lipschitz_balls_discrete}]
We prove~\eqref{eq:V_k_F_BM_n}. Consider the linear operator $A:\R^n\to\R^{n}$ defined by
$$
A(\delta_1,\ldots,\delta_n) = (\delta_1,\delta_1+\delta_2,\ldots, \delta_1+\ldots+\delta_n).
$$
Then, $\FFF_{BM,n}$ is the image of the unit cube $[0,1]^n$ under the operator $A$. In particular, $\FFF_{BM,n}$ is the parallelotope generated by the vectors $Ae_1,\ldots,Ae_n$, where $e_1,\ldots,e_n$ is the standard basis of $\R^n$. A formula for the  intrinsic volumes of a parallelotope is well-known, see~\cite[Theorem~9.8.2]{klain_rota_book}, and yields in our case
\begin{align*}
V_k(\FFF_{BM,n})
&=
\sum_{1\leq m_1<\ldots<m_k\leq n} \Vol_k (\zon (Ae_{m_1},\ldots,Ae_{m_k}))\\
&=
\sum_{1\leq l_1<\ldots<l_k\leq n} \Vol_k (\zon (Ae_{n-l_1+1},\ldots,Ae_{n-l_k+1})).
\end{align*}
Denoting by $G_{l_1,\ldots,l_k}$ the Gram matrix of the collection $\{A e_{n-l_1+1},\ldots,A e_{n-l_k+1}\}$, we have
\begin{equation}\label{eq:V_k_F_BM_n_Gram_matrix}
V_k(\FFF_{BM,n})
=
\sum_{1\leq l_1<\ldots<l_k\leq n} \sqrt{ \det (G_{l_1,\ldots,l_k})}.
\end{equation}
The $(i,j)$-th entry of $G_{l_1,\ldots,l_k}$ is given by $\min (l_{i}, l_{j})$.  The determinant of $G_{l_1,\ldots,l_k}$ can be computed by elementary row transformations, but we prefer to use probabilistic reasoning. Namely, observe that $G_{l_1,\ldots,l_k}$ is the covariance matrix of the random vector $(B(l_1),\ldots,B(l_k))$, where $B$ denotes a standard  Brownian motion. The probability density of this random vector at point zero can be computed by using the formula for the multivariate Gaussian density or by using the Markov property of the Brownian motion.
Comparing both results we obtain that
$$
\frac{1}{(\sqrt{2\pi})^{k}\sqrt {\det (G_{l_1,\ldots,l_k})}} =  \frac{1}{(\sqrt{2\pi})^{k}\sqrt{l_1(l_2-l_1)\ldots (l_k-l_{k-1})}}.
$$
Inserting the resulting formula for $\det (G_{l_1,\ldots,l_k})$ into~\eqref{eq:V_k_F_BM_n_Gram_matrix} we obtain~\eqref{eq:V_k_F_BM_n}.

\vspace*{2mm}
The proof of~\eqref{eq:V_k_F_BB_n} is similar. This time we consider the linear operator $A:\R^{n-1}\to \R^{n}$ given by
$$
A(\delta_1,\ldots,\delta_{n-1}) = \left(-s, \delta_1 - s, \delta_1+\delta_2- s,\ldots,\delta_1+\ldots+\delta_{n-1}- s\right),
$$
where
$$
s=s(\delta_1,\ldots,\delta_{n-1}) = \frac 1 {n} ((n-1)\delta_1 + (n-2) \delta_2 + \ldots+ \delta_{n-1}).
$$
Then, $\FFF_{BB,n}$ is the image of the unit cube $[0,1]^{n-1}$ under the operator $A$. By the formula for the intrinsic volumes of a parallelotope, see~\cite[Theorem~9.8.2]{klain_rota_book}, we have
\begin{equation}\label{eq:V_k_F_BB_n_Gram_matrix}
V_k(\FFF_{BB,n})
=
\sum_{1\leq l_1<\ldots<l_k\leq n-1} \sqrt{ \det (G_{l_1,\ldots,l_k})}.
\end{equation}
Here, $G_{l_1,\ldots,l_k}$ is the Gram matrix of the collection $\{Ae_{n-l_1},\ldots,Ae_{n-l_k}\}$. The $(i,j)$-th entry of this matrix is equal to $\min (l_i,l_j)- \frac 1{n} l_i l_j$. Again, it is easy to compute the determinant of $G_{l_1,\ldots,l_k}$ by using row transformations, but we will provide a probabilistic argument. Observe that $G_{l_1,\ldots,l_k}$ is the covariance matrix of the random vector $(B(l_1),\ldots, B(l_k))$ conditioned on $B(n)=0$. Computing the density of this vector by using the formula for the multivariate Gaussian density and by using the Markov property of the Brownian bridge, we obtain
$$
\frac{1}{(\sqrt{2\pi})^{k}\sqrt {\det (G_{l_1,\ldots,l_k})}} =  \frac{1}{(\sqrt{2\pi})^{k}\sqrt{l_1(l_2-l_1)\ldots (l_k-l_{k-1})(n-l_k)}}.
$$
Inserting this into~\eqref{eq:V_k_F_BB_n_Gram_matrix} we obtain~\eqref{eq:V_k_F_BB_n}.
\end{proof}

\subsection{Proof of Theorem~\ref{theo:V_k_Lipschitz_balls}}
The idea is to approximate the sets $\LLL_{BM}^{\infty}$ and $\LLL_{BB}^{\infty}$ by their discrete analogues.

\vspace*{2mm}
\noindent
\textsc{Step 1: $\LLL_{BM}^{\infty}$.} Take some $n\in\N$ and let $\LLL_{BM,n}^{\infty}$ be the parallelotope in $L^2[0,1]$ spanned by the functions $f_{1,n},\ldots,f_{n,n}$, where
$$
f_{l,n}(t) =
\begin{cases}
0, &\text{if } 0\leq t\leq \frac{n-l}{n},\\
t-\frac{n-l}{n}, &\text{if } \frac{n-l}{n}\leq t \leq \frac{n-l+1}{n},\\
1, &\text{if } \frac{n-l+1}{n}\leq t\leq 1.
\end{cases}
$$
It is clear that $\LLL_{BM,n}^{\infty}$ consists of all functions which are non-decreasing, piecewise linear with knots at $0,\frac 1n, \ldots, \frac{n-1}{n},1$, have Lipschitz constant at most $1$ and which vanish at $0$. In particular, we have $\LLL_{BM,2^n}^{\infty} \subset  \LLL_{BM,2^{n+1}}^{\infty}$, for $n\in \N$, and
$$
\LLL_{BM}^{\infty} = \overline{\bigcup_{n=1}^{\infty} \LLL_{BM, 2^n}^{\infty}}.
$$
By the lower semicontinuity of the functional $V_k$ (which is stated in Proposition~13 of~\cite{vS76} for $k=1$ but is valid for any $k\in\N$ with the same proof), we have
$$
V_k(\LLL_{BM}^{\infty}) = \lim_{n\to\infty} V_k(\LLL_{BM, 2^n}^{\infty}).
$$
To compute $V_k(\LLL_{BM, n}^{\infty})$ we proceed as in the proof of Theorem~\ref{theo:V_k_Lipschitz_balls_discrete}. The $(i,j)$-th entry of the Gram matrix of the collection $\{f_{1,n},\ldots, f_{n,n}\}$ is
$$
\langle f_{i,n}, f_{j,n} \rangle_{L^2} = \frac 1 {n^3} \left(\min (i,j) - \frac 12 -\frac 16 \ind_{i=j}\right).
$$
Thus, the Gram matrix of the collection $\{n^{3/2}f_{1,n},\ldots, n^{3/2} f_{n,n}\}$ is very close but not equal to the Gram matrix of the collection $\{A e_n,\ldots, A e_1\}$ which spans the parallelotope $\FFF_{BM,n}$. Repeating the argument from the proof of Theorem~\ref{theo:V_k_Lipschitz_balls_discrete}, we obtain
\begin{align*}
V_k(\LLL_{BM,n}^{\infty})
&=
\frac{1}{n^{3k/2}} \sum_{1\leq l_1<\ldots<l_k\leq n} \sqrt{l_1(l_2-l_1)\ldots (l_k-l_{k-1}) + O(n^{k-1})}\\
&=
\frac{1}{n^{k}} \sum_{1\leq l_1<\ldots<l_k\leq n} \sqrt{\frac{l_1}{n} \left(\frac{l_2}{n}-\frac{l_1}{n}\right) \ldots \left(\frac{l_k}{n}-\frac{l_{k-1}}{n}\right) + O(n^{-1})},
\end{align*}
where the constant in the $O$-term does not depend on $l_1,\ldots,l_k$. Replacing Riemann sums by Riemann integrals we obtain
$$
\lim_{n\to\infty} V_k(\LLL_{BM,n}^{\infty})
=
\underset{0\leq a_1<a_2<\ldots<a_k\leq 1}{\int\ldots\int}\sqrt{a_1(a_2-a_1)\ldots(a_k-a_{k-1})}\, \dd a_1\ldots \dd a_k.
$$
The integral is easy to compute:
$$
\lim_{n\to\infty} V_k(\LLL_{BM,n}^{\infty})
=
\frac 2 {3k} B\left(\frac 32 k,\ldots, \frac 32 k\right)
=
\frac{2^{-k} \pi^{k/2}}{\Gamma\left(\frac 3 2 k + 1\right)},
$$
where the Beta function $B$ has $k$ variables.  This gives the required formula for $V_k(\LLL_{BM}^{\infty})$.

\vspace*{2mm}
\noindent
\textsc{Step 2: $\LLL_{BB}^{\infty}$.}
In the setting of $\LLL_{BB}^{\infty}$ one similarly arrives at the integral
$$
V_k(\LLL_{BB}^{\infty})
=
\underset{0\leq a_1<a_2<\ldots<a_k\leq 1}{\int\ldots\int}\sqrt{a_1(a_2-a_1)\ldots(a_k-a_{k-1})(1-a_k)}\, \dd a_1\ldots \dd a_k.
$$
The integral can be computed using the Beta function with $k+1$ variables:
$$
V_k(\LLL_{BB}^{\infty})
=
B\left(\frac 32 k,\ldots, \frac 32 k\right)
=
\frac{2^{-(k+1)} \pi^{(k+1)/2}} {\Gamma\left(\frac 3 2 k + \frac 32\right)}.
$$

\vspace*{2mm}
\noindent
\textsc{Step 3: $\KKK_{BM}^{\infty}$ and $\KKK_{BB}^{\infty}$.}
To compute the intrinsic volumes of $\KKK_{BM}^{\infty}$ and $\KKK_{BB}^{\infty}$, note that with $h(t)=t/2\in L^2$ we have the set equalities
$$
\KKK_{BM}^{\infty} = 2(\LLL_{BM}^{\infty}-h),
\quad
\KKK_{BB}^{\infty} = 2(\LLL_{BB}^{\infty}-h).
$$
This implies that $V_k(\KKK_{BM}^{\infty})=2^k V_k(\LLL_{BM}^{\infty})$ and $V_k(\KKK_{BB}^{\infty})=2^k V_k(\LLL_{BB}^{\infty})$.

\begin{remark}
\citet{GV01} conjectured that for any convex $GB$-set in a Hilbert space, the sequence $m_k:=(k+1)V_{k+1}/V_k$ (which is known to be decreasing) either converges to a strictly positive limit or satisfies $m_k=O(1/\sqrt k)$. It is easy to see that for the sets $\LLL_{BM}^{\infty}$ and $\LLL_{BB}^{\infty}$ we have
$m_k\sim \text{const} /\sqrt k$, so there is no contradiction with their conjecture.
\end{remark}

\section{Gaussian width of Sobolev balls: Proof of Theorem~\ref{theo:V_1_sobolev}}\label{sec:V_1_sobolev_balls}
\subsection{The Gaussian width}
Our aim is to determine the first intrinsic volume  of $\KKK_*^p$ and $\LLL_*^p$. More generally, we will compute the distribution of the \textit{Gaussian width} of these sets.
For a bounded set $T\subset \R^n$, the Gaussian width $W_G(T)$ and the uniform width $W_U(T)$ are defined by
$$
W_G(T) = \sup_{t\in T} \langle N, t\rangle - \inf_{t\in T} \langle N, t\rangle,
\quad
W_U(T) = \sup_{t\in T} \langle U, t\rangle - \inf_{t\in T} \langle U, t\rangle,
$$
where $N$ has a standard normal distribution on $\R^n$, while $U$ has a uniform distribution on the unit sphere in $\R^n$. We have a representation
$$
W_G \eqdistr  R_n W_{U},
$$
where $R_n$ is a random variable which is independent of $U$ and such that $R_n^2$ has $\chi^2$-distribution with $n$ degrees of freedom. By the law of large numbers, $R_n/\sqrt n$ converges to $1$ in distribution, as $n\to\infty$. Thus, for large values of $n$, the scaled uniform width  $\sqrt n W_U$ is close to the Gaussian width $W_G$.  In the case of infinite $n$, the uniform width makes no sense, but there is a natural infinite-dimensional generalization of  $W_G$, namely the range of the isonormal process. Therefore,  for a set $T$ in a separable Hilbert space $H$ define its Gaussian width to be
\begin{equation}\label{eq:width_def}
\wid(T) = \sup_{t\in T} \xi(t) - \inf_{t\in T} \xi(t),
\end{equation}
where $\{\xi(h)\colon h\in H\}$ is the isonormal process over $H$. We always consider a separable version of the isonormal process; see Proposition~2.6.1 in~\cite{bogachev_book} for its existence.

The next theorem determines the Gaussian width of $\KKK_*^p$ and $\LLL_*^p$.
We use the notation $x^+=\max(x,0)$ and $x^-=\max (-x,0)$. Recall that $X_*$ is a Gaussian process as in~\eqref{eq:X_*_def_1}--\eqref{eq:X_*_def_4}.
\begin{theorem}\label{prop:mean_width_sobolev}
Let $p\in [1,\infty]$ and $\frac 1p + \frac 1q = 1$.  The maxima of the isonormal process over $\KKK_*^p$ and $\LLL_*^p$ are given by
$$
\sup_{f\in \KKK_{*}^p} \xi(f) \eqdistr  \|X_*\|_q, \quad  \sup_{f\in \LLL_{*}^p} \xi(f) \eqdistr  \|X_*^+\|_q.
$$
The Gaussian width of $\KKK_*^p$ and $\LLL_*^p$ is given by
$$
\sup_{f\in \KKK_{*}^p} \xi(f) - \inf_{f\in \KKK_{*}^p} \xi(f) \eqdistr  2\|X_*\|_q,
\quad
\sup_{f\in \LLL_{*}^p} \xi(f) - \inf_{f\in \LLL_{*}^p} \xi(f) \eqdistr \|X_*^+\|_q + \|X_*^-\|_q.
$$
\end{theorem}
Using Sudakov's Theorem~\ref{theo:sudakov} we immediately obtain Theorem~\ref{theo:V_1_sobolev} as a corollary of Theorem~\ref{prop:mean_width_sobolev}. The proof of Theorem~\ref{prop:mean_width_sobolev} will be given in Sections~\ref{subsec:isonormal} and~\ref{subsec:proof_prop_mean_width_sobolev}.

\subsection{Isonormal process over $\KKK_*^p$ and $\LLL_*^p$} \label{subsec:isonormal}
Let $\{W(t)\colon t\in [0,1]\}$ be a standard Brownian motion.   The isonormal process $\{\xi(f)\colon f\in L^2\}$ is given by
$$
\left\{\xi(f)\colon f\in L^2 \right\} \eqfdd \left\{\int_0^1 f(t)\dd W(t)\colon f\in L^2\right\},
$$
where the stochastic integral is in the It\^{o} sense.
In Lemma~\ref{lem:isonormal_alternative_rep}  below we will provide an alternative representation of the isonormal process over $\KKK_*^p$ and $\LLL_*^p$. But first we show that the embedding of $\KKK_*^p$ and $\LLL_*^p$ into $L^2$ is injective.

\begin{lemma}\label{lem:ae_equal_hence_equal_restate}
Let $p\in [1,\infty]$. If $f\in \KKK_*^p$ and $g\in \KKK_*^p$ are equal Lebesgue-a.e.\ on $[0,1]$, then they are equal everywhere on $[0,1]$ (for $p\neq 1$) or on $\R$ (for $p=1$).
\end{lemma}
\begin{proof}
In the case $p\neq 1$ the functions $f$ and $g$ are continuous, so that the statement becomes trivial. Let $p=1$. Then, the functions $f$ and $g$ are right-continuous at any $t\in [0,1)$, so they must coincide there. We have to show that $f(0-)=g(0-)$ and $f(1)=g(1)$.

\vspace*{2mm}
\noindent
\textsc{Case $*=BM$.} Then, we have the boundary condition $f(0-)=g(0-)=0$ and the functions $f$ and $g$ are left-continuous at $1$, so that $f(1)=g(1)$.

\vspace*{2mm}
\noindent
\textsc{Case $*=CBM$.} Then, we have the boundary conditions $f(0-)=g(0-)=0$ and $f(1)=g(1)=0$.

\vspace*{2mm}
\noindent
\textsc{Case $*=BB$.} Then, $f$ and $g$ are left-continuous at $1$ and hence, $f(1)=g(1)$. Also,  we know that $f(0)=g(0)$ and since $J_f(0)=J_g(0)=0$, we get $f(0-)=g(0-)$.

\vspace*{2mm}
\noindent
\textsc{Case $*=CBB$.}  We know that $f(0)=g(0)$ and $J_f(0)=J_g(0)=0$, hence $f(0-)=g(0-)$. Also, we have the boundary conditions $f(0)=f(1)$ and $g(0)=g(1)$, hence $f(1)=g(1)$.
\end{proof}

\begin{lemma}\label{lem:isonormal_alternative_rep}
With $X_*$ as in~\eqref{eq:X_*_def_1}--\eqref{eq:X_*_def_4} we have
\begin{equation}\label{eq:isonormal_repr}
\left\{\int_0^1 f(t) \dd W(t) \colon f\in \KKK_*^{p}\right\}
\eqfdd
\left\{ \int_0^1  X_{*}(1-t) \dd f(t)\colon f\in \KKK_*^{p}\right\},
\end{equation}
and similarly with $\LLL_*^p$ instead of $\KKK_*^p$.
\end{lemma}
\begin{proof}
Note that any $f\in \KKK_*^p$ is a function with bounded variation. Integrating by parts, see e.g.~\cite[Theorem~2.3.7]{kuo_book} for justification,  we have
$$
\int_0^1 f(t) \dd W(t) = f(1)W(1) - \int_0^1 W(t)\dd f(t).
$$

\begin{case}{1: $*=BM$.}
Then, the process $\{X_*(1-t)\colon t\in [0,1]\}$ (which is a standard Brownian motion with reversed time) has the same finite-dimensional distributions as $\{W(1)-W(t)\colon t\in [0,1]\}$.  We have $f(0-)=0$ and hence,
$$
\int_0^1 (W(1)-W(t)) \dd f(t) = f(1) W(1) - \int_0^1 W(t) \dd f(t) = \int_0^1 f(t) \dd W(t). 
$$
\end{case}
This proves~\eqref{eq:isonormal_repr}.

\vspace*{2mm}
In the remaining three cases, the process $\{X_*(1-t)\colon t\in [0,1]\}$ has the same finite-dimensional distributions as the process $\{-X_*(t) \colon t\in [0,1]\}$.
So, we need to prove that
\begin{equation}\label{eq:isonormal_repr_modified}
\left\{\int_0^1 f(t) \dd W(t) \colon f\in \KKK_*^{p}\right\}
\eqfdd
\left\{ - \int_0^1  X_{*}(t) \dd f(t)\colon f\in \KKK_*^{p}\right\}.
\end{equation}

\begin{case}{2: $*=CBM$.}
We have $f(0-)=f(1)=0$ and writing $N=\int_0^1 W(s) \dd s$ we obtain
$$
-\int_0^1 X_*(t)\dd f(t)
=
\int_0^1 (N - W(t)) \dd f(t)
=
- \int_0^1 W(t) \dd f(t)
=
\int_0^1 f(t) \dd W(t).
$$
\end{case}

\begin{case}{3: $*=BB$.}
We have $\int_0^1 f(t) \dd t = 0$ and hence, by definition of $X_*(t)$,
$$
-\int_0^1 X_*(t)\dd f(t)
=
W(1)\int_0^1 t \dd f(t) - \int_0^1 W(t) \dd f(t)
=
f(1)W(1) - \int_0^1 W(t) \dd f(t),
$$
where we used that $\int_0^1 t \dd f(t) = f(1)$ by integration by parts.
\end{case}

\begin{case}{4: $*=CBB$.}
We have $f(0)=f(1)$  and $\int_0^1 f(t)\dd t =0$. Writing $N = \int_0^1 (W(s) -sW(1)) \dd s$ we have
$$
-\int_0^1 X_*(t)\dd f(t)
=
\int_0^1 (tW(1) - W(t) + N) \dd f(t)
=
f(1)W(1) - \int_0^1 W(t) \dd f(t),
$$
where we used that $\int_0^1 \dd f(t)=0$ and $\int_0^1 t\dd f(t) = f(1)$ by integration by parts.
\end{case}
In the setting of $\LLL_*^p$, the statement of the lemma follows by restriction from $\KKK_*^p$.
\end{proof}

\begin{lemma}\label{lem:continuity}
Let $p\in [1,\infty]$ and  $X:[0,1]\to\R$ be a continuous function. In the case $p=1$ we make the following additional assumptions:
\begin{enumerate}
\item If $*=BM$, then $X(1)=0$.
\item If $*=BB$, then $X(0)=X(1)=0$.
\item If $*=CBB$, then $X(0)=X(1)$.
\end{enumerate}
Then,
$$
\Psi: f\mapsto \int_0^1 X(t)\dd f(t)
$$
is a continuous mapping from $\KKK_*^p$ or $\LLL_*^p$ (considered as subsets of $L^2$) to $\R$.
\end{lemma}
\begin{remark}
Consequently, the right-hand side of~\eqref{eq:isonormal_repr} defines a Gaussian process with continuous sample paths. (Note that the process $X(t):=X_*(1-t)$ satisfies the boundary conditions of Lemma~\ref{lem:continuity}). Thus, the sets $\KKK_*^p$ and $\LLL_*^p$ have the GC-property. In the sequel, we always deal with the version of the isonormal process over $\KKK_*^p$ or $\LLL_*^p$ which is given by the right-hand side of~\eqref{eq:isonormal_repr}.
\end{remark}
\begin{remark}
Let us consider an example showing that the assumptions on $X(0)$ and $X(1)$ in the case $p=1$ cannot be omitted. Consider the sequence
$$
f_n(t) = \ind_{[1-\frac 1n,\infty)}(t) \in \KKK_{BM}^1.
$$
It converges in $L^2[0,1]$ to $0$. For a continuous function $X$ not satisfying $X(1)=0$ we would have
$$
\lim_{n\to\infty} \Psi(f_n)
=
\lim_{n\to\infty} X\left(1-\frac 1n\right)
=  X(1) \neq 0=\Psi(0).
$$
Similar examples are possible for other values of $*$.
\end{remark}

\begin{proof}[Proof of Lemma~\ref{lem:continuity}]
It suffices to prove the following statement:  For arbitrary $f,f_1,f_2,\ldots \in \KKK_{*}^p$ such that  $f_n\to f$ in the $L^2$-sense as $n\to\infty$,  there is a subsequence $f_{n_i}$ for which $\Psi(f_{n_i})$ converges to $\Psi(f)$, as $i\to\infty$.

\vspace*{2mm}
\noindent
\textsc{Step 1.}
We prove that it is possible to find a subsequence $f_{n_i}$ for which $\Psi(f_{n_i})$ converges to \textit{some} limit. Note that the total variation of the function $f_n$ is bounded by $1$ for every $n$.
For $p=1$ this follows from the definition of $\KKK_*^{1}$, whereas for $p>1$ we have $f_n\in AC[0,1]$ and $TV(f_n)=\|f'_n\|_1 \leq \|f'_n\|_p\leq 1$  by the Lyapunov inequality and the definition of $\KKK_*^p$. The inequality $TV(f_n)\leq 1$, together with the boundary conditions, implies that $\|f_n\|_{\infty} \leq 1$.

We can introduce the signed Lebesgue--Stieltjes measures $\mu_n((s,t]) = f_n(t)-f_n(s)$, $s<t$. The total variation of $\mu_n$ is at most $1$. By Helly's theorem, we can extract a subsequence $\mu_{n_i}$ converging weakly to some signed measure $\mu$, as $i\to\infty$. Note that $\mu$ is concentrated on the interval $[0,1]$. If $p\neq 1$, then we can tell more.  Namely, for every $0\leq x\leq y\leq 1$ by the H\"older inequality we have
$$
|f_n(x)-f_n(y)| = \left|\int_x^y f_n'(s) \dd s\right|
\leq
\|f_n'\|_p \, |y-x|^{1/q}
\leq |y-x|^{1/q}.
$$
By the Arzel\`{a}--Ascoli theorem we can extract a subsequence $f_{n_i}$ which converges uniformly to some continuous function. It follows that $\mu_{n_i}$ converges weakly to some signed measure $\mu$ which has no atoms.

Since $X$ is a continuous function, it follows from the definition of weak convergence that
$$
\lim_{i\to\infty} \Psi(f_{n_i}) = \lim_{i\to\infty}\int_0^1 X(t) \mu_{n_i}(\dd t) = \int_0^1 X(t) \mu(\dd t).
$$


\vspace*{2mm}
\noindent
\textsc{Step 2.}
We prove that $\Psi(f)=\int_0^1 X(t)  \mu(\dd t)$.

\vspace*{2mm}
\noindent
In the proof below we consider the case $p=1$.
The proof in the case $p\neq 1$ is similar and, in fact, even much simpler, because in this case the measure $\mu$ has no atoms and therefore we can ignore terms with $\mu(\{0\})$ and $\mu(\{1\})$.

\vspace*{2mm}
\noindent
\textsc{Case $*=BM$.} Define a measure $\mu^{\circ} = \mu  - \delta_1 \mu(\{1\})$, where $\delta_1$ is the delta-measure at $1$. Consider the function $h(t)=\mu^{\circ}((-\infty, t])$. By construction, $h$ is c\'adl\'ag and $h(0-)=0$, $J_h(1)=1$, so that $h$ satisfies the same boundary conditions as the functions from $\KKK_{BM}^1$.  By the definition of the weak convergence, we have
$$
f_{n_i}(t) = \mu_{n_i}((-\infty, t])\to \mu((-\infty, t]) = h(t), \text{ as } i\to\infty,
$$
for all $t\in (0,1)$ where $h$ is continuous.
By the dominated convergence theorem, we obtain that $f_{n_i}$ converges to $h$ in $L^2$. On the other hand, $f_{n_i}$ converges to $f$ in $L^2$. By the uniqueness of the $L^2$-limit, $f$ and $h$ coincide a.e.\ on $[0,1]$. By the same reasoning as in Lemma~\ref{lem:ae_equal_hence_equal_restate}, these functions in fact  coincide everywhere. It follows that
$$
\Psi(f) = \Psi(h)  = \int_0^1 X(t) \dd h(t) = \int_0^1 X(t) \mu(\dd t) - X(1)\mu(\{1\})
=
\int_0^1 X(t) \mu(\dd t),
$$
where the last step holds because we have the assumption $X(1)=0$.

\vspace*{2mm}
\noindent
\textsc{Case $*=CBM$.} Define a function $h(t)=\mu((-\infty, t])$. We have the boundary condition $f_{n_i}(0-)=f_{n_i}(1)=0$ implying that $\mu_{n_i}([0,1])=0$ and hence, $\mu([0,1])=0$. This implies that $h(0-)=h(1)=0$. Also, $h$ is c\'adl\'ag. So, $h$ satisfies the same boundary conditions as the functions from $\KKK_{CBM}^1$. By the definition of the weak convergence, we have
$$
f_{n_i}(t)=\mu_{n_i}((-\infty, t])\to \mu((-\infty, t]) = h(t), \text{ as } i\to\infty,
$$
for all $t\in (0,1)$ where $h$ is continuous. By the dominated convergence theorem, we obtain that $f_{n_i}$ converges to $h$ in $L^2$. On the other hand, $f_{n_i}$ converges to $f$ in $L^2$. By the uniqueness of the $L^2$-limit, $f$ and $h$ coincide a.e.\ on $[0,1]$. By the same reasoning as in Lemma~\ref{lem:ae_equal_hence_equal_restate}, these functions in fact  coincide everywhere. Hence,
$$
\Psi(f) = \Psi(h)  = \int_0^1 X(t) \dd h(t) = \int_0^1 X(t) \mu(\dd t),
$$
as required.

\vspace*{2mm}
\noindent
\textsc{Case $*=BB$.} Define $\mu^{\circ} = \mu - \delta_0 \mu(\{0\}) - \delta_1 \mu(\{1\})$, where $\delta_0$ and $\delta_1$ are delta-measures at $0$ and $1$. Consider the function $h(t)=\mu^{\circ}((-\infty, t]) + c$, where $c$ is a constant chosen such that $\int_0^1 h(t) \dd t = 0$. By construction, $h$ is c\'adl\'ag and satisfies the boundary conditions $J_h(0)=J_h(1)=0$. By the definition of weak convergence, we have $\mu_{n_i}((-\infty, t]) \to \mu ((-\infty, t])$ for every  $t\in (0,1)$ where $h$ is continuous. It follows that with constants $c_{n_i}=f_{n_i}(0-) - c + \mu(\{0\})$ we have $f_{n_i}(t) - c_{n_i} \to h(t)$ for every  $t\in (0,1)$ where $h$ is continuous. Note that the sequence $c_{n_i}$ is bounded. By the dominated convergence theorem, $f_{n_i}-c_{n_i}$ converges to $h$ in $L^2$. On the other hand, $f_{n_i}$ converges to $f$ in $L^2$. It follows that $h-f$ is constant a.e.\ on $[0,1]$. However, since both $f$ and $h$ have vanishing integral over $[0,1]$, they coincide a.e. By the reasoning of Lemma~\ref{lem:ae_equal_hence_equal_restate}, $f$ and $h$ coincide everywhere. It follows that
\begin{align*}
\Psi(f) = \Psi(h)  = \int_0^1 X(t) \dd h(t)
&= \int_0^1 X(t) \mu(\dd t)- X(0)\mu(\{0\})-  X(1)\mu(\{1\})\\
&=\int_0^1 X(t) \mu(\dd t),
\end{align*}
where the last step holds because we have the assumption $X(0)=X(1)=0$.

\vspace*{2mm}
\noindent
\textsc{Case $*=CBB$.} Define a measure $\mu^{\circ} = \mu + \mu(\{0\}) (\delta_1 - \delta_0)$. Consider a function $h(t)=\mu^{\circ}((-\infty, t]) + c$, where $c$ is a constant chosen such that $\int_0^1 h(t) \dd t = 0$. We have the boundary condition $f_{n_i}(0-)=f_{n_i}(1)=0$ which implies that  $\mu_{n_i}([0,1])=0$ and hence, $\mu^{\circ}([0,1])=\mu([0,1])=0$. By construction, $h$ is c\'adl\'ag and satisfies the boundary conditions $h(1)=h(0-)$ and $J_h(0)=0$.
By the definition of the weak convergence, $\mu_{n_i}((-\infty,t]) \to \mu((-\infty,t])$ for every  $t\in (0,1)$ where $h$ is continuous.  Defining the constants $c_{n_i}=f_{n_i}(0-)-c+\mu(\{0\})$, we have $f_{n_i}(t) - c_{n_i} \to h(t)$ for every  $t\in (0,1)$ where $h$ is continuous. By the dominated convergence theorem, $f_{n_i}(t) - c_{n_i}$ converges to $h(t)$ in $L^2$. On the other hand, $f_{n_i}$ converges to $f$ in $L^2$. It follows that $f(t)-h(t) = c$ a.e.\ for  a suitable constant $c\in \R$. However, since  both $f$ and $h$  have vanishing integral over $[0,1]$, we have $f=h$ a.e.\ on $[0,1]$ and, by the reasoning of Lemma~\ref{lem:ae_equal_hence_equal_restate}, even everywhere on $\R$. It follows that
\begin{align*}
\Psi(f) = \Psi(h)  = \int_0^1 X(t) \dd h(t)
&= \int_0^1 X(t) \mu(\dd t) + (X(1)-X(0))\mu(\{0\})\\
&=\int_0^1 X(t) \mu(\dd t),
\end{align*}
where we used the assumption $X(0)=X(1)$.
\end{proof}

\subsection{Proof of Theorem~\ref{prop:mean_width_sobolev}}\label{subsec:proof_prop_mean_width_sobolev}
Recall that by Lemma~\ref{lem:isonormal_alternative_rep} the isonormal process is given by
$$
\{\xi(f)\colon f\in \KKK_*^p\} = \left\{ \int_0^1 X_*(1-t) \dd f(t) \colon f\in \KKK_*^p\right\}.
$$
Let first $p\in (1,\infty]$.
Then, any $f\in \KKK_*^p$ is absolutely continuous. By the H\"older inequality we have
\begin{equation}\label{eq:wspom_ineq_1}
\left|\int_0^1 X_{*}(1-t) \dd f(t)\right| = \left|\int_0^1 f'(t) X_{*}(1-t) \dd t\right| \leq \|f'\|_p \|X_*\|_q \leq \|X_*\|_q.
\end{equation}
On the other hand, the equality in~\eqref{eq:wspom_ineq_1} is attained if $f=g$, where
$$
g(t) :=
\begin{cases}
\int_0^t \left(\frac{X_*(1-s)}{\|X_*\|_q}\right)^{q-1} \dd s, &\text{if } p\in (1,\infty),\\
\int_0^t \sgn X_*(1-s) \dd s, &\text{if } p=\infty.
\end{cases}
$$
Below we will show that it is possible to modify $g$ such that it satisfies the boundary conditions of $\KKK_*^p$.

But let us first consider the case $p=1$. Then, the total variation of every $f\in \KKK_{*}^{1}$ is at most $1$ and hence,
$$
\left|\int_0^1 X_{*}(1-t) \dd f(t)\right| \leq
\sup_{t\in[0,1]} |X_*(t)|.
$$
The equality is attained if $f=g$, where
$$
g(t) =
\begin{cases}
0, & t <  \arg\max |X_*(1-\cdot)|,\\
1, & t \geq  \arg\max |X_*(1-\cdot)|.
\end{cases}
$$

Let us now show how to modify the minimizer $g$ to make the boundary conditions satisfied. Let $p\in (1,\infty]$.
\begin{case}{$*=BM$.}
Choose $f=g$ since the boundary condition $g(0)=0$ is satisfied.
\end{case}

\begin{case}{$*=BB$.}
Choose $f(t)=g(t)+ a$, where $a$ is a constant  such that $\int_0^1 f(s)\dd s=0$.
\end{case}

\begin{case}{$*=CBM$.}
Choose $f(t)=g(t)+ a + b t$, where $a,b$ are constants  such that $f(0)=f(1)=0$. Note that $\int_0^1 X_*(t)\dd t=0$ (since $X_*$ is the centered Brownian motion) and hence,
\begin{equation}\label{eq:tech1_boundary_cond}
\int_0^1 X_{*}(1-t) \dd f(t)= \int_0^1 X_{*}(1-t) \dd g(t).
\end{equation}
\end{case}

\begin{case}{$*=CBB$.}
Choose $f(t)=g(t)+ a + bt$, where $a,b$ are constants such that $f(0)=f(1)$ and $\int_0^1 f(s)\dd s = 0$.
Note that $\int_0^1 X_*(t)\dd t=0$ (since $X_*$ is the centered Brownian bridge) and hence, \eqref{eq:tech1_boundary_cond} holds.
\end{case}
For $p=1$ the argument is the same, but we have also to note that $J_g(0)=J_g(1)=0$ by definition (since the process $X_*$ does not attain its maximum at $0$ or at $1$). So, the boundary conditions of $\KKK_*^1$ are satisfied.

\vspace*{2mm}
Let us now consider the maximum over $\LLL^p_{*}$. Since every $f\in \LLL^p_{*}$ is monotone non-decreasing, we have
$$
\left|\int_0^1 X_{*}(1-t) \dd f(t)\right| \leq  \left|\int_0^1 \max(X_{*}(1-t),0) \dd f(t)\right|
$$
and one can repeat the same considerations as in cases $*=BM$ and $*=BB$ above with $X_*$ replaced by $\max (X_*,0)$.
\hfill $\Box$

\section{Intrinsic volumes of ellipsoids in Hilbert space}\label{sec:V_1_hilbert_ellipsoids}
\subsection{The first intrinsic volume of an ellipsoid}
Consider a separable Hilbert space $H$ over $\R$ with an orthonormal basis $\psi_1,\psi_2,\ldots$. For concreteness, we assume that $H$ is infinite-dimensional, but the same considerations apply in the finite-dimensional case.
Let $\lambda_1,\lambda_2,\ldots$ be a sequence of positive numbers such that $\sum_{n=1}^{\infty} \lambda_n^2 <\infty$. Consider the following subset of $H$:
\begin{equation}\label{eq:def_dispersion_ellipsoid}
\EE = \left\{h = \sum_{n=1}^{\infty} x_n \psi_n \in H \colon \sum_{n=1}^\infty\frac{x_n^2}{\lambda_n^2}\leq 1\right\}.
\end{equation}
Note that $\EE$ is an ellipsoid with half-axes $\lambda_1, \lambda_2, \ldots$.
Let us derive a formula for the Gaussian width and the first intrinsic volume of $\EE$.
\begin{proposition}\label{prop:mean_width_ellipsoid}
Consider the random variable  $M:=\sum_{n=1}^{\infty} \lambda_n^2 N_n^2$, where $N_1,N_2,\ldots$ are i.i.d.\ standard normal random variables. Then, the Gaussian width and the first intrinsic volume of the ellipsoid $\EE$ defined in~\eqref{eq:def_dispersion_ellipsoid} are given by
\begin{equation}\label{eq:V_1_E_sqrt_M}
\wid(\EE)\eqdistr 2\sqrt M,
\quad
V_1(\EE)
=\sqrt{2\pi} \, \E \sqrt{M}.
\end{equation}
\end{proposition}
\begin{remark}
We have $M<\infty$ a.s.\ since we assume that $\sum_{n=1}^{\infty} \lambda_n^2 <\infty$. Hence, the set $\EE$ is a GB-set.
\end{remark}
\begin{proof}[Proof of Proposition~\ref{prop:mean_width_ellipsoid}]
The isonormal process  $\{\xi(h): h\in H\}$ is given  as follows: For $h = \sum_{n=1}^{\infty} x_n\psi_n\in H$ we have
$$
\xi(h) = \sum_{n=1}^{\infty} x_n N_n.
$$
By the Cauchy--Schwarz inequality and the definition of $\EE$, see~\eqref{eq:def_dispersion_ellipsoid}, we have the estimate
\begin{equation}\label{eq:sup_isonormal_ineq}
\sup_{h\in \EE} \xi(h)
=
\sup_{h \in \EE} \sum_{n=1}^{\infty}\left(\lambda_n N_n  \cdot \frac{x_n}{\lambda_n}\right) \leq  \left(\sum_{n=1}^{\infty} \lambda_n^2 N_n^2\right)^{1/2}=\sqrt M.
\end{equation}
On the other hand, for $x_n=\lambda_n^2 N_n/\sqrt M$ an equality is attained in~\eqref{eq:sup_isonormal_ineq}, so that
$$
\sup_{h\in \EE} \xi(h)
=\sqrt M.
$$
Thus, by Sudakov's formula, the first intrinsic volume of $\EE$ is given by~\eqref{eq:V_1_E_sqrt_M}.
\end{proof}
\begin{remark}\label{rem:rivin}
\citet{rivin} obtained a formula very similar to Proposition~\ref{prop:mean_width_ellipsoid} for the surface area (which is $2V_{n-1}$) of the ellipsoid. Namely, he showed that for an $n$-dimensional ellipsoid $\EE^*$ with half-axes $1/\lambda_1,\ldots, 1/\lambda_n$, the surface area is given by
\begin{equation}\label{eq:rivin_surface_area_ellipsoid}
2V_{n-1} (\EE^*) = \frac {\sqrt 2} {\lambda_1\ldots \lambda_n} \frac{\pi^{n/2}}{\Gamma\left(\frac{n+1}{2}\right)} \E \sqrt{\lambda_1^2 N_1^2 + \ldots +\lambda_n^2 N_n^2}.
\end{equation}
In fact, Rivin's formula~\eqref{eq:rivin_surface_area_ellipsoid} can be deduced from Proposition~\ref{prop:mean_width_ellipsoid}, see Proposition~\ref{prop:rivin} below. The results on the surface area obtained in Rivin's paper~\cite{rivin} can be translated to the setting of $V_1$.
\end{remark}

\subsection{Special cases: $E_d$ and $F_d$}\label{subsec:V_1_hilbert_special_cases}
We now consider some special cases  in which it is possible to compute $\E \sqrt{M}$ explicitly. Using the formula  $\E \eee^{-tN_n^2} = (1+2t)^{-1/2}$ we obtain that the Laplace transform of $M$ is given by
$$
\E \eee^{-t M} = \prod_{n=1}^{\infty} (1 + 2\lambda_n^2 t)^{-1/2}.
$$

\begin{example} Denote by $E_d$ the ellipsoid whose half-axes are $\frac{1}{n \pi}$, $n\in\N$,
where each value has the same multiplicity $d\in \N$.
Then, the Laplace transform of $M$ is given by
\begin{equation}\label{eq:laplace_trans_S_t}
\E \eee^{-t M}
=
\prod_{n=1}^{\infty} \left( 1 + \frac{2t}{n^2\pi^2} \right)^{-d/2}
=
\left(\frac{\sqrt {2t}}{\sinh \sqrt{2t}}\right)^{d /2}.
\end{equation}
Random variables with Laplace transform~\eqref{eq:laplace_trans_S_t} appear frequently in probability theory and were studied in~\cite{biane_etal} and~\cite{pitman_yor}. A generic random variable $M$ with Laplace transform~\eqref{eq:laplace_trans_S_t} is denoted by $S_{d/2}$ in these papers where, among many other results, some moments of $S_1$ and $S_2$ were calculated.  By Proposition~\ref{prop:mean_width_ellipsoid}, the width of $E_d$ is
\begin{equation}\label{eq:width_ellipsoid_S_t}
\wid (E_d) \eqdistr 2 \sqrt{S_{d/2}}.
\end{equation}
It follows from Proposition~\ref{prop:mean_width_ellipsoid} and the results of~\cite{biane_etal} (see, e.g.,\ Table~1  in~\cite{pitman_yor}) that
\begin{equation}\label{eq:mean_width_ellipsoid_S_t}
V_1(E_d) = \sqrt{2\pi} \, \E \sqrt{S_{d/2}} =
\begin{cases}
2\log 2, &\text{ if } d=2,\\
2, &\text{ if } d=4.
\end{cases}
\end{equation}
\end{example}

\begin{example} Denote by $F_d$ the ellipsoid whose half-axes are $\frac{1}{(n-\frac 12)\pi}$, $n\in\N$,
where each value has the same multiplicity $d\in\N$.
Then, the Laplace transform of $M$ is given by
\begin{equation}\label{eq:laplace_trans_C_t}
\E \eee^{-t M}
=
\prod_{n=1}^{\infty} \left( 1 + \frac{2t}{\left(n-\frac 12\right)^2\pi^2} \right)^{-d/2}
=
\left(\frac{1}{\cosh \sqrt{2t}}\right)^{d/2}.
\end{equation}
A generic random variable $M$ with Laplace transform~\eqref{eq:laplace_trans_C_t} was denoted by $C_{d/2}$ in~\cite{biane_etal} and~\cite{pitman_yor}. By Proposition~\ref{prop:mean_width_ellipsoid}, the width of $F_d$ is
\begin{equation}\label{eq:width_ellipsoid_C_t}
\wid (F_d) \eqdistr 2 \sqrt{C_{d/2}}.
\end{equation}
From the formulae for $\E C_{d/2}^{1/2}$ derived in~\cite{biane_etal} (see, e.g.,\ Table~1 in~\cite{pitman_yor}) we obtain that
\begin{equation}\label{eq:mean_width_ellipsoid_C_t}
V_1(F_d) = \sqrt{2\pi} \, \E \sqrt{C_{d/2}} =
\begin{cases}
\frac{8 G}{\pi}, &\text{ if } d=2,\\
\frac{28}{\pi^{2}} \zeta(3), &\text{ if } d=4.
\end{cases}
\end{equation}
Here, $G=\sum_{n=0}^{\infty} \frac{(-1)^n}{(2n+1)^2}$ is the Catalan constant and $\zeta(3)=\sum_{n=1}^{\infty} \frac{1}{n^3}$.
\end{example}

\subsection{Arbitrary intrinsic volumes of ellipsoids}
Proposition~\ref{prop:mean_width_ellipsoid} can be generalized to higher intrinsic volumes as follows.
\begin{proposition}\label{prop:mean_width_ellipsoid_general}
For every $k\in\N$, the $k$-th intrinsic volume of the ellipsoid $\EE$ defined in~\eqref{eq:def_dispersion_ellipsoid} is given by
$$
V_k (\EE)
=
\frac{(2\pi)^{k/2}}{k!} \E\sqrt{\det W_k},
$$
where $W_k$ is a random $k\times k$-matrix whose $(i,j)$-th entry equals $\sum_{n=1}^\infty\lambda_n^2 N_{n,i} N_{n,j}$, and $\{N_{n,i}: n\in\N, 1\leq i\leq k\}$ are i.i.d.\ standard normal random variables.
\end{proposition}
\begin{proof}
Recall that $\psi_1,\psi_2,\ldots$ is an orthonormal basis of $H$ and that we represent a vector $h\in H$ in the form $h=\sum_{n=1}^{\infty} x_n\psi_n$.  Define $k$ independent isonormal processes $\{\xi_i(h): h\in \EE\}$, where $1\leq i\leq k$, by
$$
\xi_i(h) = \sum_{n=1}^{\infty} x_n N_{n,i}.
$$
Keeping in mind  Tsirelson's Theorem~\ref{theo:tsirelson_spectrum}, consider the random convex set
$$
\spec_k \EE =\{(\xi_{1}(h), \ldots, \xi_{k}(h))\colon h\in \EE\}\subset  \R^k.
$$
Define a column vector $y=y(h)\in \R^{\infty}$ and a $k\times\infty$ matrix $A$ by
$$
y=\left(\frac{x_1}{\lambda_1},\frac{x_2}{\lambda_2},\ldots\right)^{T}\in \R^{\infty},
\quad
A=\left(\lambda_n N_{n,i}\right)_{n\in \N, 1\leq i\leq k}.
$$
Then, $\|y\|_2\leq 1$ if and only if $h\in \EE$ and we have a representation
$$
\spec_k \EE = \{A y \colon y\in \R^{\infty}, \|y\|_2\leq 1\}.
$$
Denote by $a_1,\dots,a_k$ the row vectors of the matrix $A$:
$$
a_i=\left(\lambda_1 N_{1,i},\lambda_2 N_{2,i},\dots\right)\in \ell^2 \;\;\;\text{a.s.},\quad 1\leq i\leq k,
$$
and let $V$ be the linear span of $\{a_1,\dots,a_k\}$ in the Hilbert space $\ell^2$ of square summable sequences. It holds that $\dim V=k$ and $V^{\bot}=\text{Ker} A$ a.s.  Therefore,
$$
\spec_k \EE = \{A y \colon y\in V, \|y\|_2\leq 1\}.
$$
Any $y\in V$ a.s.\ has a unique representation $y=c_1a_1+\dots+c_ka_k$ and it holds that
$$
\|y\|^2_2=\sum_{i,j=1}^kc_ic_j\langle a_i,a_j\rangle=\langle AA^{T}c,c\rangle, \quad Ay=AA^Tc,
$$
where $c=(c_1,\dots,c_k)^{T}$. It follows that
$$
\spec_k \EE = \{AA^Tc \colon c\in \R^k, \langle AA^{T}c,c\rangle\leq 1\}
= \{x\in \R^k \colon \langle (AA^{T})^{-1}x,x\rangle\leq 1\}.
$$
Thus, $\spec_k \EE$ is an ellipsoid defined by the quadratic form  $AA^{T}$. The volume of an ellipsoid is known (see, e.g., Proposition~\ref{prop:rivin} below for $k=n$), and we obtain
$$
\Vol_k (\spec_k \EE) = \kappa_k \sqrt{\det (A A^{T})} = \kappa_k \sqrt {\det W_k}.
$$
The proof is completed by applying Tsirelson's Theorem~\ref{theo:tsirelson_spectrum}.
\end{proof}
\begin{remark}
In the finite-dimensional case, Proposition~\ref{prop:mean_width_ellipsoid_general} was obtained in \cite{KZ12}.
\end{remark}

The next proposition states a duality between $V_k$ and $V_{n-k}$ for ellipsoids. It explains Remark~\ref{rem:rivin}.  Let $\Sigma$ be a symmetric, positive definite $n\times n$ matrix. Consider the following two ellipsoids in $\R^n$:
$$
\EE = \{x\in \R^n\colon \langle x, \Sigma^{-1} x\rangle \leq 1\},\;\;\;
\EE^* = \{x\in \R^n\colon \langle x, \Sigma x\rangle \leq 1\}.
$$
If the half-axes of $\EE$ are $\lambda_1,\ldots,\lambda_n$, then the half-axes of $\EE^*$ are $1/\lambda_1,\ldots,1/\lambda_n$.
\begin{proposition}\label{prop:rivin}
For every $0\leq k\leq n$ it holds that
$$
V_k(\EE) = |\det \Sigma|^{1/2} \frac{\kappa_k}{\kappa_{n-k}} V_{n-k} (\EE^*).
$$
\end{proposition}
\begin{proof}
It is known that there exists a symmetric positive definite matrix denoted by $U$ such that $\Sigma^{-1}=U^2$.
Let $B_n$ be the unit ball in $\R^n$. Then, considering $U$ as a linear operator on $\R^n$, we have
$$
U(\EE) = B_n, \;\;\; U(B_n) = \EE^*.
$$
Intrinsic volumes are a special case of mixed volumes, see, e.g.,~\cite[Section~14.2]{schneider_weil_book}:
\begin{equation}
V_k(\EE)
= \frac{\binom{n}{k}}{\kappa_{n-k}} V(\underbrace{\EE,\ldots,\EE}_k, \underbrace{B_n,\ldots,B_n}_{n-k}).
\end{equation}
Applying the linear transformation $U$ to the mixed volumes, we obtain
\begin{align*}
V_k(\EE)
&= |\det U|^{-1} \frac{\binom{n}{k}}{\kappa_{n-k}} V(\underbrace{U(\EE),\ldots,U(\EE)}_k, \underbrace{U(B_n),\ldots,U(B_n)}_{n-k})\\
&= |\det \Sigma|^{1/2} \frac{\binom{n}{k}}{\kappa_{n-k}} V(\underbrace{B_n,\ldots,B_n}_k, \underbrace{\EE^*,\ldots,\EE^*}_{n-k})\\
&= |\det \Sigma|^{1/2} \frac{\kappa_k}{\kappa_{n-k}} V_{n-k}(\EE^*).
\end{align*}
This is the desired formula.
\end{proof}


\section{Special cases $p=1,2,\infty$}

\subsection{The Gaussian width of $\KKK_*^1$ and $\LLL_*^1$}\label{subsec:BV_Balls_V1}
Here we consider the case $p=1$. Recall that, essentially,  the set $\KKK_{*}^{1}$ consists of functions on $[0,1]$ whose total variation is bounded by $1$, with additional boundary or integral conditions. In the set $\LLL_*^1$ we additionally require the functions to be monotone non-decreasing. Applying Theorem~\ref{prop:mean_width_sobolev} (with $q=\infty$) and~\eqref{eq:width_def} we obtain the distribution of the Gaussian width of $\KKK_{*}^1$ and $\LLL_*^1$:
\begin{align}
&\wid (\KKK_{BM}^1) \eqdistr 2\sup_{t\in[0,1]} |W(t)|  \eqdistr \frac{2}{\sqrt{C_1}},\label{eq:wid_1}\\
&\wid (\LLL_{BM}^1) \eqdistr \sup_{t\in[0,1]} W(t) - \inf_{t\in[0,1]} W(t) \eqdistr \frac{2}{\sqrt{C_2}},\label{eq:wid_2}\\
&\wid (\KKK_{BB}^1) \eqdistr 2\sup_{t\in [0,1]} |X_{BB}(t)|  \eqdistr \pi \sqrt{S_1},\label{eq:wid_3}\\
&\wid (\LLL_{BB}^1) \eqdistr \sup_{t\in[0,1]} X_{BB}(t) - \inf_{t\in[0,1]} X_{BB}(t) \eqdistr \frac{\pi} 2 \sqrt{S_2},\label{eq:wid_4}
\end{align}
where the known characterizations of the distribution of the supremum  and the range of the Brownian motion and the Brownian bridge in terms of the distributions $S_1,S_2,C_1,C_2$ were used; see~\cite{biane_etal}. Note that $\frac 12 \wid(\KKK_{BB}^1)$ has the Kolmogorov--Smirnov distribution, whereas $\wid(\LLL_{BB}^1)$ has the limiting distribution of the Kuiper's test.   By comparing~\eqref{eq:wid_3} and~\eqref{eq:wid_4} with~\eqref{eq:width_ellipsoid_S_t}, we obtain the following distributional identities
$$
\wid (\KKK_{BB}^1) \eqdistr \frac {\pi}{2}\wid (E_2),
\quad
\wid (\LLL_{BB}^1) \eqdistr \frac {\pi}{4}\wid (E_4).
$$
Trying to explane these strange coincidences, one may conjecture that there is an isometry between the corresponding sets. As a support of this conjecture, one can show that
$$
\diam (\KKK_{BB}^1) = \frac {\pi}{2} \diam  (E_2) = \frac 12,
\quad
\diam (\LLL_{BB}^1) = \frac {\pi}{4} \diam (E_4) = \frac 14,
$$
where $\diam (T) = \sup_{x,y\in T} \|x-y\|_2$.
However, the conjecture is not true.
\begin{proposition}\label{prop:no_isometry}
Equipped with the $L^2$-metric, the sets $\KKK_{BB}^1$ and $\frac {\pi}{2} E_2$ are not isometric. Similarly, the sets $\LLL_{BB}^1$ and $\frac {\pi}{4} E_4$ are not isometric.
\end{proposition}
\begin{proof}
Suppose that there is an isometry $\varphi$ between $\LLL_{BB}^1$ and $\frac {\pi}{4} E_4$. From the isometric property of $\varphi$ it follows that it must be affine, that is $\varphi(tx+(1-t)y)=t\varphi(x)+(1-t)\varphi(y)$ for all $x,y\in \LLL_{BB}^1$ and $t\in [0,1]$.  In particular, $\varphi$ is the homeomorphism between the sets of extreme points of the convex sets $\LLL_{BB}^1$ and $\frac {\pi}{4} E_4$ endowed with the induced $L^2$-topology. The extreme points of $\LLL_{BB}^1$  are the functions
$$
f_{\alpha}(t) = (\alpha-1) \ind_{(-\infty,\alpha)}(t) + \alpha \ind_{[\alpha,\infty)}(t), \quad \alpha\in (0,1),
$$
and the zero function. To see this, note that $\LLL_{BB}^1$ is the image of the convex set $S$ of all (non-negative) measures $\mu$ on $(0,1)$ with  $\mu((0,1))\leq 1$ under the map $A$ which maps $\mu\in S$ to the function $t\mapsto \mu([0,t]) - \int_0^1 \mu([0,s])\dd s$. The extreme points of $S$ are the Dirac measures $\delta_{\alpha}$, $\alpha\in (0,1)$, and the zero measure. Since $A$ is affine and bijective, the extreme points of $\LLL_{BB}^1$ are $f_{\alpha}=A\delta_{\alpha}$ and $0$. Note that $A\delta_{\alpha}\to 0$ (in $L^2$) as $\alpha\to 0$ or $\alpha\to 1$.

So, the set of extreme points of $\LLL_{BB}^1$ is homeomorphic to the circle. The set of extreme points of $\frac {\pi}{4} E_4$ is the boundary of $\frac {\pi}{4} E_4$.
Clearly, these sets of extreme points are not homeomorphic (one of them is infinite-dimensional while the other is not), thus proving the absence of isometry between $\LLL_{BB}^1$ and $\frac {\pi}{4} E_4$.

Similarly, the extreme points of the convex set $\KKK_{BB}^1$ are the functions $f_{\alpha}$, $-f_{\alpha}$, $\alpha\in (0,1)$. Again, there is no homeomorphism between the sets of extreme points of $\KKK_{BB}^1$ and $\frac {\pi}{2} E_2$.
\end{proof}

\begin{example}
Recall Sudakov's Theorem~\ref{theo:sudakov}:
$$
V_1(T)= \sqrt{ \frac {\pi}{2} }\, \E \wid(T).
$$
Applying this to~\eqref{eq:wid_1}--\eqref{eq:wid_4} and using the identities
$$
\sqrt{2\pi}\, \E C_1^{-1/2} = \pi,\;\;
\sqrt{2\pi}\, \E C_2^{-1/2} = 2,\;\;
\sqrt{2\pi}\, \E S_1^{1/2} = 2\log 2,\;\;
\sqrt{2\pi}\, \E S_2^{1/2} = 2,
$$
see~\cite[Equation~(56)]{pitman_yor} and~\eqref{eq:mean_width_ellipsoid_S_t}, we obtain the formulae for the first intrinsic volumes stated in Proposition~\ref{prop:V1_p_1}.
\end{example}

\subsection{The first intrinsic volume of $\KKK_*^2$}
By Theorem~\ref{theo:V_1_sobolev}, the first intrinsic volume of $\KKK_*^2$ can be related to the expected $L^2$-norm of the process $X_*$ as follows:
$$
V_1(\KKK_*^2) = \sqrt {2\pi}\, \E \left(\int_{0}^1 X_*^2(t)\dd t\right)^{1/2}.
$$
The distribution of the squared $L^2$-norm of $X_*$ has been much studied (see, e.g.,~\cite{biane_etal}). Using the Karhunen--Loeve expansion of the Gaussian process $X_*$ it can be expressed as the weighted $\chi^2$-distribution with weights which are characterized in terms if the eigenvalues of the covariance operator of $X_*$. In our cases, the distribution of the squared $L^2$-norm is of the form $S_d$ or $C_d$; see Section~\ref{subsec:V_1_hilbert_special_cases}. Moreover, using the same method we will show that $\KKK_*^2$ is isometric to an ellipsoid of the form $E_d$ or $F_d$.

Let us first introduce a $d$-dimensional generalization of $\KKK_*^2$ as follows. Denote by $AC^d[0,1]$ the set of absolutely continuous functions $f:[0,1]\to\R^d$. Define
$$
\KKK^{2,d} =  \left\{f=(f_1,\ldots,f_d)\in AC^d[0,1]\colon f_1',\ldots,f_d'\in L^2[0,1], \sum_{i=1}^d \|f_i'\|_2^2\leq 1\right\}.
$$
Then, define $\KKK^{2,d}_*$ for all admissible values of $*$ by imposing on each component $f_1,\ldots,f_d$ the same boundary conditions as in Section~\ref{sec:sobolev_balls}. Note that $\KKK_{BM}^{2,d}$ is the Strassen ball (of the $d$-dimensional Brownian motion) which appears for example in the functional law of the iterated logarithm.
\begin{proposition}\label{prop:sobolev_ellips_p_2}
In the $L^2$-metric,
\begin{enumerate}
\item $\KKK_{BM}^{2,d}$ is isometric to $F_{d}$;
\item $\KKK_{BB}^{2,d}$ and $\KKK_{CBM}^{2,d}$ are isometric to $E_d$;
\item $\KKK_{CBB}^{2,d}$ is isometric to $\frac 12 E_{2d}$.
\end{enumerate}
\end{proposition}
\begin{proof}
The proof uses characterization of Sobolev balls with $p=2$ in terms of Karhunen--Loeve expansions. Let first $d=1$.
\begin{case}{$*=BM$.} Every real-valued function $\psi \in L^2$ has an orthonormal expansion of the form
$$
\psi(t)=\sum_{k=1}^{\infty} a_k \sqrt 2 \sin\left(\left(k-\frac 12\right)\pi t\right).
$$
This is an expansion in terms of the eigenfunctions of the Laplace operator with boundary conditions $f(0)=f'(1)=0$. We have $\psi\in \KKK_{BM}^2$ if and only if
$$
\sum_{k=1}^{\infty} \pi^2 \left(k-\frac 12\right)^2 a_k^2 \leq 1,
$$
thus establishing  the isometry between $\KKK_{BM}^2$ and $F_1$.
\end{case}

\begin{case}{$*=CBM$.} We can write any function $\psi\in L^2$ in the form
$$
\psi(t)=\sum_{k=1}^{\infty} a_k \sqrt 2 \sin (k\pi t).
$$
These are the eigenfunctions of the Laplace operator with Dirichlet boundary conditions $f(0)=f(1)=0$. We have $\psi\in \KKK_{CBM}^2$ if and only if
$
\sum_{k=1}^{\infty} \pi^2 k^2 a_k^2 \leq 1,
$
thus showing that  $\KKK_{CBM}^2$ is isometric to $E_1$.
\end{case}

\begin{case}{$*=BB$.}
We can write any function $\psi\in L^2$ with vanishing integral in the form
$$
\psi(t)=\sum_{k=1}^{\infty} a_k \sqrt 2 \cos (k\pi t).
$$
These are the eigenfunctions of the Laplace operator with Neumann boundary conditions $f'(0)=f'(1)=0$. We have $\psi\in \KKK_{BB}^2$ if and only if
$
\sum_{k=1}^{\infty} \pi^2 k^2 a_k^2 \leq 1,
$
thus showing that  $\KKK_{BB}^2$ is isometric to $E_1$.
\end{case}

\begin{case}{$*=CBB$.} 
We can expand any function $\psi\in L^2$ with vanishing integral into a Fourier series
$$
\psi(t)=\sum_{k=1}^{\infty} (a_k \sin (2\pi k t) + b_k \cos (2\pi k t)).
$$
These are the eigenfunctions of the Laplace operator with periodic boundary condition $f(0)=f(1)$.  We have $\psi\in \KKK_{CBB}^2$ if and only if
$
\sum_{k=1}^{\infty} \pi^2 k^2 (a_k^2 + b_k^2) \leq \frac 14,
$
thus showing that  $\KKK_{CBB}^2$ is isometric to $\frac 12 E_2$.
\end{case}

In the case of  arbitrary $d\in\N$ one has to expand the components of the function $\psi$ separately.
\end{proof}

\begin{example}
Using Proposition~\ref{prop:sobolev_ellips_p_2} together with~\eqref{eq:mean_width_ellipsoid_S_t}, \eqref{eq:mean_width_ellipsoid_C_t}, we obtain
\begin{align*}
&V_1(\KKK_{BB}^{2,2}) = V_1(\KKK_{CBM}^{2,2}) = 2 V_1(\KKK_{CBB}^{2,1}) = 2\log 2,\\
&V_1(\KKK_{BB}^{2,4}) = V_1(\KKK_{CBM}^{2,4}) = 2 V_1(\KKK_{CBB}^{2,2}) = 2,\\
&V_1(\KKK_{BM}^{2,2}) = \frac{8G}{\pi},\\
&V_1(\KKK_{BM}^{2,4}) = \frac{28}{\pi^2}\zeta(3).
\end{align*}
\end{example}

\subsection{The first intrinsic volume of $\KKK_*^{\infty}$ and $\LLL_{*}^{\infty}$}\label{subsec:Lipschitz_V1}
Here we consider the case $p=\infty$. Recall that the sets $\KKK_*^{\infty}$ consist of functions which have Lipschitz constant at most $1$ and are subject to additional boundary conditions. In the set $\LLL_*^{\infty}$ the functions are additionally required to be monotone.  Applying Theorem~\ref{prop:mean_width_sobolev} with $q=1$ and using the notation $\sigma_*^2(t)=\Var X_*(t)$ we obtain the formula
$$
V_1(\KKK_*^{\infty}) = \sqrt {2\pi}\, \E \int_0^1 |X_*(t)| \dd t = \sqrt {2\pi}\, \E|N| \int_0^1 \sigma_*(t) \dd t  =  2\int_0^1 \sigma_*(t)\dd t.
$$
Here, $N$ has the standard normal distribution and we used the fact that $\E |N| =\sqrt{2/\pi}$. For the first intrinsic volume of $\LLL_*^{\infty}$ we obtain
$$
V_1(\LLL_*^{\infty}) = \sqrt {2\pi}\, \E \int_0^1 |X_*^+(t)| \dd t = \sqrt {2\pi}\, \E \max(N,0) \int_0^1 \sigma_*(t) \dd t  = \int_0^1 \sigma_*(t)\dd t.
$$
The variance  $\sigma_*^2(t)$ is given by
$$
\sigma_{BM}^2(t) = t, \;\;
\sigma_{BB}^2(t) =t(1-t), \;\;
\sigma_{CBM}^2(t)=t^2-t+\frac 13,\;\;
\sigma_{CBB}^2(t)=\frac 1 {12}.
$$
Evaluating the integral of $\sigma_*(t)$ we obtain  the formulae for  the  first intrinsic volume of $\KKK_{*}^{\infty}$ stated in Proposition~\ref{prop:V1_p_infty}.  Similarly, we obtain that  the first intrinsic volume of $\LLL_*^{\infty}$ is given by
\begin{align*}
V_1(\LLL_{BM}^{\infty}) = \frac 23,
\quad
V_1(\LLL_{BB}^{\infty}) = \frac {\pi}{8}.
\end{align*}

\subsection*{Acknowledgement} The authors are grateful to Wolfgang Arendt and Markus Kunze for a discussion related to Proposition~\ref{prop:no_isometry}.


\bibliographystyle{plainnat}
\bibliography{sobolev_ellipsoids_bib}

\end{document}